\documentclass[11pt]{amsart}

\usepackage{mathptmx}

\usepackage{graphicx}
\usepackage{amsmath,amsfonts,amsthm,amssymb,amscd}
\usepackage{hyperref}
\usepackage[margin=1in]{geometry}
\usepackage{enumitem}

\usepackage{color}

\newcommand{\ie}{{i.e.,\ }}
\newcommand{\eg}{{e.g.,\ }}

\DeclareMathOperator*{\esssup}{ess\,sup}
\DeclareMathOperator*{\essinf}{ess\,inf}
\DeclareMathOperator{\var}{Var}
\DeclareMathOperator{\cov}{Cov}
\DeclareMathOperator{\neigh}{ne}
\DeclareMathOperator{\cl}{cl}

\newcommand\be{\begin{equation}}
\newcommand\ee{\end{equation}}
\newcommand\bea{\begin{eqnarray}}
\newcommand\eea{\end{eqnarray}}
\newcommand\bi{\begin{itemize}}
\newcommand\ei{\end{itemize}}
\newcommand\ben{\begin{enumerate}}
\newcommand\een{\end{enumerate}}

\newcommand{\Z}{\ensuremath{\mathbb{Z}}}
\newcommand{\N}{\mathbb{N}}
\newcommand{\E}{\mathbb{E}}
\renewcommand{\P}{\mathbb{P}}

\newtheorem{thm}{Theorem}[section]
\newtheorem*{thm*}{Theorem}
\newtheorem{cor}[thm]{Corollary}
\newtheorem{lem}[thm]{Lemma}
\newtheorem{prop}[thm]{Proposition}

\theoremstyle{definition}
\newtheorem{defi}[thm]{Definition}
\newtheorem*{defi*}{Definition}
\newtheorem{exa}[thm]{Example}

\theoremstyle{remark}
\newtheorem{rek}[thm]{Remark}

\newcommand{\ci}{\perp\!\!\!\perp}

\numberwithin{equation}{section}
\begin{document}
\title[Graphical Markov Models for Infinitely Many Variables]{Graphical Markov Models for Infinitely Many Variables}
\date{\today}

\author[David Montague and Bala Rajaratnam]{David Montague and Bala Rajaratnam \\ Stanford University}

\thanks{D.M. and B.R. are partially supported by the following: US Air Force Office of Scientific Research grant award FA9550-13-1-0043, US National Science Foundation under grant DMS-0906392, DMS-CMG 1025465, AGS-1003823, DMS-1106642, DMS-CAREER-1352656, Defense Advanced Research Projects Agency DARPA YFA N66001-111-4131, the UPS Foundation, SMC-DBNKY, and the National Science Foundation Graduate Research Fellowship under Grant No.\ DGE-114747}
\address{Departments of Mathematics and Statistics, Stanford
University, Stanford, CA - 94305, USA}
\email{D.M. davwmont@gmail.com; B.R. brajaratnam01@gmail.com}
\date{\today}

\begin{abstract}
Representing the conditional independences present in a multivariate random vector via graphs has found widespread use in applications, and such representations are popularly known as graphical models or Markov random fields. These models have many useful properties, but their fundamental attractive feature is their ability to reflect conditional independences between blocks of variables through graph separation, a consequence of the equivalence of the pairwise, local and global Markov properties demonstrated by Pearl and Paz (1985). Modern day applications often necessitate working with either an infinite collection of variables (such as in a spatial-temporal field) or approximating a large high-dimensional finite stochastic system with an infinite-dimensional system. However, it is unclear whether the conditional independences present in an infinite-dimensional random vector or stochastic process can still be represented by separation criteria in an infinite graph. In light of the advantages of using graphs as tools to represent stochastic relationships, we undertake in this paper a general study of infinite graphical models. First, we demonstrate that na{\"i}ve extensions of the assumptions required for the finite case results do not yield equivalence of the Markov properties in the infinite-dimensional setting, thus calling for a more in-depth analysis. To this end, we proceed to derive general conditions which \emph{do} allow representing the conditional independence in an infinite-dimensional random system by means of graphs, and our results render the result of Pearl and Paz as a special case of a more general phenomenon. We conclude by demonstrating the applicability of our theory through concrete examples of infinite-dimensional graphical models. 
\end{abstract}
\maketitle

\setcounter{equation}{0}
\setcounter{tocdepth}{1}
\vspace{-0.25in} \tableofcontents

\section{Introduction}

\subsection{Background}
Representing the conditional independences present in a multivariate random vector via graphs has found widespread use in applications, and has also led to important theoretical advances in probability. Such probabilistic Markov models are popularly known as graphical models or Markov random fields, and have seen rapid development in recent years. Though the initial motivation for their development primarily stemmed from statistical physics \cite{Snell}, they have become particularly relevant in modern applications. In particular, graphical models are widely used to provide easily interpreted graphical representations of complex multivariate dependencies that are present in a large collection of random variables. Thus, they are very valuable for analyzing modern high-dimensional throughput data, and have become staples in contemporary statistics, computer science, and allied fields. 

In such models, the nodes of a graph correspond to random variables. The absence of an edge between a pair of nodes represents conditional independence between the corresponding two random variables given all the other variables, and is known as the pairwise Markov property. A particularly useful feature of graphical models is the equivalence of the so-called pairwise (P), local (L), and global (G) Markov properties. These Markov properties relate a graph $\mathcal{G} = (V,E)$ and a collection of random variables $(X_i)_{i \in V}$ as follows:
\bi
\item (P) For $i,j\in V$, $i \not \sim_{\mathcal{G}} j$ implies $X_i \ci X_j | X_{V\setminus \{i,j\}}$.
\item (L) Given $i \in V$, $X_i \ci X_{V \setminus \cl(i)} | X_{\neigh(i)}$.
\item (G) Given $A, B \subseteq V$, and $C \subseteq V$ separating $A$ and $B$, we have $X_A \ci X_B | X_C$.
\ei
Here we have used the following notation: 1) for $A \subseteq V$, $X_A = (X_a)_{a \in A}$, 2) $i \sim_{\mathcal{G}} j$ means nodes $i$ and $j$ are adjacent in $\mathcal{G}$, 3) $\neigh(i)$ is the neighbor set of $i$ in $\mathcal{G}$ (so that $\neigh(i) = \{ j \in V \ | \ i \sim_\mathcal{G} j \}$), and 4) $\cl(i) = \{i\} \cup \neigh(i)$.

Whereas the pairwise Markov relationship represents independences between a pair of nodes, the global Markov property allows one to infer from the graph conditional independences between two blocks of variables, given a third block, if the first two blocks are ``separated" by this third block in the graph. 

One of the cornerstones of the field of graphical models is the equivalence of these three Markov properties for finite graphical models, which holds under relatively mild conditions and was first demonstrated by Pearl and Paz \cite{PearlPaz}. This equivalence result provides a way to relate the local and global conditional independence structure within a collection of random variables, and establishes the value of the graphical representation of these models. However, this equivalence between the pairwise, local, and global Markov properties does not readily extend to infinite collections of random variables. 

There are both theoretical and practical reasons for studying infinite-dimensional graphical models. First, the study of the limits of sequences of graphs has seen much development in recent years \cite{Lovasz2006, graphs1, graphs2}, yet the ability of these ``infinite graphs'' to represent multivariate dependencies in an infinite-dimensional random vector is not well understood. Second, modern applications have to contend with very high-dimensional data, and approximating such large finite systems with an infinite system has some clear practical advantages. For instance, by the same broad principles relevant when using general asymptotic or limiting approximations to better understand or analyze finite models, an infinite-dimensional graphical model can also serve to better illustrate the salient features of a large finite counterpart. We expand on these and other compelling reasons for studying infinite-dimensional graphical models in subsection \ref{sec:motivation}.

The goal of this paper is therefore to understand the equivalence of the global, local, and pairwise Markov properties for an infinite collection of variables. We begin by first introducing a graphical framework which allows us to establish the equivalence of analogues of the pairwise, local, and global Markov properties for ternary relations on infinite graphs from a purely set-theoretic perspective (Section \ref{sec:AGF}). We then develop the probability theory required to apply this graphical framework to the relation induced by conditional independence on  a countable collection of random variables (Section \ref{sec:CI}). The set-theoretic framework is introduced prior to developing the relevant probability theory in order to separate the difficulties arising from infinite graphs versus infinite probability distributions. Finally, we demonstrate the broad applicability of our results through two general classes of graphical stochastic processes: a) Gaussian processes (Section \ref{sec:GGP}), and b) discrete (i.e., $\{0,1\}$-valued) processes (Section \ref{sec:discproc}). More specifically, we obtain sufficient conditions for the equivalence of the Markov properties in each of these contexts, and subsequently verify them on a collection of examples. These examples include, among others, autoregressive processes in the Gaussian setting, and an infinite-dimensional extension of the Ising model in the discrete setting. For these examples, we also actually verify one of the Markov properties, allowing us to immediately conclude the other two by equivalence, and thus establish the collection of variables as an infinite-dimensional graphical model. 

The proofs of the various results contained in this paper are placed either in the main text, the Appendix, or the Supplemental section depending on the centrality of the result and so that they do not excessively detract from the flow in the main body of the paper.

\subsection{Motivation} \label{sec:motivation}

One of the primary reasons for the popularity of graphical models in modern applications stems from their ability to facilitate the inference of conditional independences between large sets of variables  (as opposed to conditional independence between just two variables). This is achieved through the global Markov property by using separation statements in graphs, and relies fundamentally on the equivalence of the global, local, and pairwise Markov properties. The Pearl and Paz \cite{PearlPaz} result relating the various Markov properties in the finite setting has already been very useful for these purposes, but the more general infinite case has yet to receive a similar treatment. We provide a number of compelling reasons why such an investigation is important and long overdue, both from a theoretical and application perspective:

\begin{enumerate}
\item Due to the nature of modern data science and ``Big Data'' applications, the number of variables or features under consideration is becoming increasingly very large. In fact, the number of variables can easily be in the hundreds of thousands, or even many millions (e.g., gene expression data, EQTL high throughput data, single nucleotide polymorphism data, remote sensing data, high frequency trading data, etc.). The verification of the equivalence of the Markov properties by checking that the joint density is positive is simply not feasible for many modern applications for at least two reasons. First, note that the enormous number of variables is often accompanied by limited samples. This sample starved setting leads to extremely sparsely distributed data in very high-dimensional space. In such settings traditional frequency curves or histograms are not useful for describing the underlying data generating mechanisms or probability models. Second, the very large number of variables also means that the shape and support of each of the marginal densities cannot be easily checked, let alone the shape and support of the joint density. While the conditions for equivalence of the Markov properties for infinite collections of random variables established in this paper do imply the existence of a positive density, they may in some cases serve as a more directly verifiable safeguard for the very high-dimensional but still finite setting, thus still allowing one to invoke the global Markov property. 

\item Many collections of random variables are by their very nature infinite-dimensional. For instance, many collections of random variables evolve in time and/or space with either a countable or uncountable index set. Though such processes may be observed only at finite times due to practical constraints, they are in principle actually infinite.  At least in discrete examples of such settings (e.g., infinite lattice based models), an understanding of the conditional independences between variables that are present in the true underlying random process can only be achieved by examining the infinite analogues of the Markov properties.

\item Consider the case of the standard Gaussian graphical model. In the finite-dimensional setting it is well known that graphical structure corresponds to classes of covariance matrices with zeros in the inverse. The families of distributions of inferential interest here are thus Gaussian distributions parameterized by sparse inverse covariance matrices. We shall show in this paper that in order to obtain graphical structure in the infinite setting, further conditions are required on infinite-dimensional covariance matrices that parameterize Gaussian processes. In particular, our investigation into infinite-dimensional probabilistic models identifies classes of probability measures which enjoy graphical structure. Establishing and identifying classes of such infinite graphical models therefore serves the important purpose of laying the probabilistic foundations required for statistical inference. It specifies clearly the families of measures on which inference is to be undertaken if one is interested in obtaining probabilistic data generating mechanisms which describe observed data in a parsimonious way.  

\item Recall that theoretical guarantees of traditional statistical inferential methods are proven in the ``classical asymptotic regime" when the number of variables (or dimension) $p$ remains fixed and the sample size $n$  tends to infinity. The advent of high throughput data, especially in the last two decades, has seen inferential techniques for estimating or recovering a sparse graphical model in the ``mixed asymptotic regime"  when both sample size $n$ and number of variables $p$ tend to infinity. Very recently, a new class of methods for recovering sparse graphical models has been proposed in the ``purely high-dimensional asymptotic regime" when the sample size $n$ is actually fixed but the dimension $p$ tends to infinity \cite{HeroRajaratnam2011, HeroRajaratnam2012}.  This regime has been argued as the appropriate regime for modern ``Big Data" applications and corresponds to the infinite-dimensional setting.  Establishing theoretical safeguards of inferential procedures in the ``purely high-dimensional"  asymptotic regime, however, requires conditions on both the joint distributions as well as sparsity in parameters of interest \cite{HeroRajaratnam2011, HeroRajaratnam2012}. These technical conditions are motivated only by statistical considerations, but it is not immediately clear if these inferential procedures actually yield true probabilistic graphical models in the infinite setting (i.e., in which the Global Markov property can be invoked). The work in this paper thus also serves to address this gap by undertaking a probabilistic treatment of infinite graphical models.

\item Another compelling motivation for our work is an important technical one. Recall that the equivalence of the Markov properties in the finite setting relies fundamentally on the so-called intersection property. We demonstrate in this paper that a na{\"i}ve or straightforward extension of the intersection property to infinite sets, and the verification thereof, does not ensure the equivalence of the global, local, and pairwise Markov properties. Thus the ability to infer conditional independences between blocks of variables in the infinite setting is simply not guaranteed using assumptions from the finite-dimensional case. The verification of such assumptions may give the false impression that the global Markov property holds when it actually does not, motivating a rigorous treatment of the infinite-dimensional setting.
\end{enumerate}

\section{Summary of Main Results}

We now provide an overview of the main results in the paper. Our first goal is to obtain equivalence of the pairwise, local, and global Markov properties for ternary relations. First, we provide the necessary definitions. 

\begin{defi} \label{def:p1p5}
Let $V$ be any set. Let $\cdot \perp \cdot \ | \ \cdot $  be a ternary relation on the power set $\mathcal{P}(V)$. The following properties of the relation are satisfied if their respective statements are true for all (potentially infinite) sets $X, Y, Z, W \in \mathcal{P}(V)$:

\bi
\item (P1*) \emph{Symmetry:} $X \perp Y \ | \ Z$ implies $Y \perp X \ | \ Z$;
\item (P2*) \emph{Decomposition:} $X \perp (Y,W) \ | \ Z$ implies $X \perp Y \ | \ Z$ and $X \perp W \ | \ Z$;
\item (P3*) \emph{Weak Union:} $X \perp (Y,W) \ | \ Z$ implies $X \perp Y \ | \ (Z,W)$;
\item (P4*) \emph{Contraction:} $X \perp Y \ | \ (Z,W)$ and $X \perp W \ | \ Z$ imply $X \perp (Y,W) \ | \ Z$;
\item (P5*) \emph{General Intersection:} For any partition $X = \cup_k X_k$ into \emph{finite} subsets of $V$, if $X_k \perp Y \ | \ Z \cup \left (\bigcup_{j \neq k} X_j\right )$ for all $k$, then $X \perp Y \ | \ Z$.
\ei

Following the nomenclature in \cite{Lauritzen}, we call the relation $\cdot \perp \cdot \ | \ \cdot$ a \emph{semi-graphoid relation} on $V$ provided it satisfies (P1*) -- (P4*). If in addition the relation satisfies (P5*), we call it an \emph{extended graphoid relation}. 
\end{defi}

While (P1*)--(P4*) are essentially the same as their traditional finite counterparts which have already appeared in the literature, the statement of (P5*) is a less straightforward generalization of the usual intersection property to the infinite setting. We also note that the cardinality of the various sets in the statements of (P1*)--(P5*) are bounded only by $|V|$, i.e., if $|V|$ is uncountably infinite then the various sets $X, Y, Z, W$ may also be.

\begin{defi} \label{defi:smp}
If $V$ is any set, $\mathcal{G}$ is any undirected graph on $V$, and $\cdot \perp \cdot \ | \ \cdot$ is any ternary relation on $\mathcal{P}(V)$, we say that $(V,\perp)$ satisfies the pairwise (P*), local (L*), or global (G*) \emph{set-Markov properties} with respect to $\mathcal{G}$ provided that 
\bi
\item (P*) for each pair of vertices $\{i,j\} \subseteq V$,
$$i \not \sim_{\mathcal{G}} j \Rightarrow \{i\} \perp \{j\} \ | \ V \setminus \{i,j\};$$
\item (L*) for each vertex $i$,
$$\{i\} \perp V \setminus \cl(i) \ | \ \neigh(i);$$
\item (G*) for each triple (A,B,S) of disjoint subsets of $V$ with $S$ separating $A$ from $B$, we have
$$A \perp B \ | \ S.$$
\ei

We say that (P*), (L*), or (G*) \emph{holds over} $\mathcal{G}$ when $(V,\perp)$ is clear from context.
\end{defi}

We prove the following theorem, which in particular establishes the equivalence of the set-Markov properties under (P1*)--(P5*), and in the case of finite $V$ reduces to the foundational result of Pearl and Paz \cite{PearlPaz}.
\begin{thm} \label{thm:CIREquiv}
Suppose $V$ is any set, and $\cdot \perp \cdot \ | \ \cdot $ satisfies (P1*) -- (P4*). Then
$$(V,\perp) \textrm{ satisfies (G*) } \Rightarrow (V,\perp) \textrm{ satisfies (L*) }\Rightarrow (V,\perp) \textrm{ satisfies (P*)}.$$
If, in addition, $\cdot \perp \cdot \ | \ \cdot$ satisfies (P5*), then we also have
$$(V,\perp) \textrm{ satisfies (P*) } \Rightarrow (V,\perp) \textrm{ satisfies (G*)}.$$
\end{thm}

After establishing the set theoretic framework outlined above, we apply Theorem \ref{thm:CIREquiv} to the ternary relation induced by conditional independences present in a countably infinite collection of random variables.

\begin{cor}
Suppose that $\{X_n\}_{ n \in \mathbb{N}}$ is a collection of random variables satisfying property (P5*). Then for the ternary relation $\cdot \ci_X \cdot \ | \ \cdot$ on $\mathbb{N}$ induced by the conditional independences in the collection of random variables $\{X_n\}_{ n \in \mathbb{N}}$,
$$(\mathbb{N},\ci_X) \textrm{ satisfies (G*) } \Leftrightarrow (\mathbb{N},\ci_X) \textrm{ satisfies (L*) }\Leftrightarrow (\mathbb{N},\ci_X) \textrm{ satisfies (P*)}.$$
\end{cor}

The above corollary establishes the importance of (P5*), but leaves its verification untouched. Our next goal is to establish the equivalence of the various Markov properties for a wide class of examples of infinite collections of random variables by verifying (P5*). To this end, we introduce in the main text sufficient conditions for (P5*) to hold for a general collection of random variables (we give these sufficient conditions the names ``infinite intersection property'' (IIP) and ``decorrelation property'' (DCP); see Definition \ref{def:d1d2} for their statements). We then verify these conditions for two important and widely used classes of models: countably infinite Gaussian processes, and countably infinite discrete processes. Our primary results in this regard are given below.

\begin{thm} \label{thm:gaussianMR}
Let $(X_n)_{n \in \mathbb{N}}$ be a Gaussian process with covariance matrix $\Sigma$ which satisfies the following bounds:
\ben
\item There exist constants $c$ and $C$ such that for any finite subset of nodes $A$, we have
$$0 < c < \lambda_i(\Sigma_A) < C,$$
for all $i$, where $\lambda_i(\Sigma)$ represents the $i^{\text{th}}$ largest eigenvalue of the matrix $\Sigma$.
\item There exists a function $g_0(i,j)$ which bounds the covariances
$$|\cov(X_i, X_j)| \leq g_0(i,j) \ \ \ \forall i,j$$
and if we recursively define
$$g_{n+1}(i,j) = g_n(i,j) + \sum_{k=1}^\infty g_n(i,k) g_n(k,j),$$ then
for $n \leq 4$, $g_n(i,j)$ exists and is finite, and moreover
$$\sum_{k=1}^\infty g_n(i,k) k^\epsilon < \infty$$ for some $\epsilon > 0$ and all $i$.
\een
Then the ternary relation $\cdot \ \ci_X \ \cdot \ | \ \cdot$ satisfy the infinite intersection property (IIP) and the decorrelation property (DCP), and hence (P5*).
\end{thm}

\begin{rek}
We comment briefly on the intuitive interpretation of conditions (1) and (2) above. The eigenvalue condition (1) ensures sufficient level of nonsingularity of the distribution that it is possible to obtain a positive density, which will be important for verifying (IIP). Condition (2) can be interpreted as a ``decorrelation'' condition, as it forces $|\cov(X_i, X_j)|$ to decay sufficiently fast as $X_i$ and $X_j$ become further apart, and will be important for verifying (DCP). One commonly used rich class of models for which this second condition holds is the lattice model with the exponential or Gaussian covariance function \cite{tingley}. Full details for the Gaussian covariance function are included in Example \ref{exa:cov}.
\end{rek}

\begin{rek}
We note that the conditions on the $g_n$ from (2) in Theorem \ref{thm:gaussianMR} are conceptually similar to ``decorrelation'' conditions required for other probabilistic results for dependent random variables. As a first example, the central limit theorem for dependent variables requires
$$\lim_{n \to \infty} \frac{1}{n}\sum_{i=1}^n \sum_{j \neq i}^n \cov(X_i,X_j) = \gamma$$
for some finite constant $\gamma$ \cite{lehmann}. This condition is similar to that on the $g_n$ from Theorem \ref{thm:gaussianMR} in that both ultimately imply a certain level of decay in the covariances of the variables in question. A second example of a useful decorrelation condition is the often imposed assumption of absolute summability of the autocorrelations in a covariance-stationary sequence of random variables:
$$\sum_{k=0}^\infty |\cov(X_0,X_k)| := \sum_{k=0}^\infty |\gamma_k| < \infty.$$ 
This condition can be used to obtain a weak law of large numbers \cite{doob}, and is actually implied by the second part of condition (2) in Theorem \ref{thm:gaussianMR}. 
\end{rek}

\begin{thm} \label{thm:discreteMR}
Suppose the random variables $\{X_n \ | \ n \in \mathbb{N}\}$ form a $\{0,1\}$-valued stochastic process satisfying the following conditions:
\ben
\item For any finite increasing sequences of natural numbers $I = (i_1,...,i_m)$ and $J = (j_1,...,j_r)$, there is a constant $c_I$ depending only on $I$ such that for any two $\{0,1\}$-valued sequences $(a_1,...,a_m)$ and $(b_1,...,b_r)$, we have
$$\mathbb{P}((X_{i_1},...,X_{i_m}) = (a_1,...,a_m) \ | \ X_{j_1}=b_1,...,X_{j_r}=b_r) > c_I \textrm{ a.s. } X_J$$
\item Let $n \in \mathbb{N}$ and $B \subseteq \mathbb{N}$ be arbitrary, and let $\mathcal{F}_{B,-n} = \sigma(X_B,X_n,X_{n+1},...)$. For any $m \in \mathbb{N}$ and $\mathbb{P}$-a.e.\ value of $X_B = x_B$, there is a function $g_{m,B,x_B}(n)$ satisfying
\bi 
\item $\lim_{n \rightarrow \infty}g_{m,B,x_B}(n) = 0$
\item For any $A \in \mathcal{F}_{B,-n}$, 
\be \var(\mathbb{P}(A|X_1,...,X_m,X_B)|X_B=x_B) < g_{m,B,x_B}(n). \label{ln:varvar} \ee

\ei
\een
Then the $\{X_n \ | \ n \in \mathbb{N}\}$ satisfy the infinite intersection property (IIP) and the decorrelation property (DCP), and hence (P5*).
\end{thm}

\begin{rek}
The quantity $\mathbb{P}(A|X_1,...,X_m,X_B)$ is a random variable depending on $X_1, ..., X_m$ and $X_B$, so the conditional variance from line \eqref{ln:varvar} is a sensible quantity.
\end{rek}

\begin{rek}
The conditions (1) and (2) from Theorem \ref{thm:discreteMR} serve a similar role to the conditions from Theorem \ref{thm:gaussianMR}. In particular, condition (1) again ensures ``nonsingularity'' of the distribution, permitting the existence of a positive density. Also, condition (2) corresponds once again to a quantitative formulation of a ``decorrelation'' condition. More specifically, it requires that if after conditioning on a set of variables $X_B$ an event $A$ only depends on variables sufficiently far from $X_1,...,X_m$, then $A$ is ``nearly independent'' of $X_1,...,X_m$.
\end{rek}

For each of the two theorems above we provide multiple examples for which the respective required conditions are satisfied. These examples will serve to demonstrate the flexibility of our extended framework for the study of infinite-dimensional graphical models.

\section{Preliminaries} \label{sec:background}

\subsection{Graph Theory}
\begin{defi}
An \emph{undirected graph} $\mathcal{G}$ is a pair of objects $(V,E)$, where $V$ is the set of vertices of $\mathcal{G}$, and $E$ is a subset of $V \times V$ containing the edges. By convention, $E$ does not contain edges of the form $(i,i)$ or multiple edges between any two vertices, and we do not distinguish between the edges $(i,j)$ and $(j,i)$. We now introduce a number of additional definitions related to undirected graphs.

\bi
\item Let $i$ be a vertex in $V$. A vertex $j$ is called \emph{adjacent} to $i$, or a \emph{neighbor} of $i$, if $(i,j) \in E$. If $j$ is a neighbor of $i$, we write $i \sim_{\mathcal{G}} j$, and otherwise we write $i \not \sim_{\mathcal{G}} j$. The set of all neighbors of $i$ is denoted by $\neigh(i)$. We also define the \emph{closure} set $\cl(i) :=  \{i\} \cup \neigh(i)$.
\item Given a graph $\mathcal{G} = (V,E)$ and a subset $A \subseteq V$, the \emph{induced subgraph} on $A$ is defined to be the graph $(A,E \cap (A \times A))$. 
\item The \emph{degree} of a vertex $i \in V$ is the cardinality of the neighbor set of $i$. That is, the degree of a vertex $i$ is equal to 
$$\deg(i) = |\{j \in V \ | \ i \sim j \}|.$$
\item Let $A,B,C \subseteq V$ be three nonempty subsets of $V$. We say that $C$ \emph{separates} $A$ from $B$ if every path from a vertex $a \in A$ to a vertex $b \in B$ contains some vertex in $C$.
\ei

\end{defi}

\subsection{Conditional Independence} \label{subsec:CI}

\begin{defi} \label{def:condIndep}
Suppose $X,Y,$ and $Z$ are random variables over the probability space $\mathcal{X}$. Suppose also that the joint probability distribution of $(X,Y,Z)$ has a density $f$ with respect to some underlying measure $\mu$ on the range of $(X,Y,Z)$. Then the variables $X$ and $Y$ are said to be conditionally independent given $Z$, denoted $X \ci Y \ | \ Z$, if and only if there is a factorization
$$f_{(X,Y)|Z}(x,y,z) = f_{X|Z}(x,z)f_{Y|Z}(y,z)$$
which holds for a.e.\ $x$, $y$, and $z$. See \cite{Dawid2} for more details.
\end{defi}

Consider a collection of random variables $\{X_v \ | \ v \in V\}$. Given a subset $A \subseteq V$ let $X_A$ denote the random vector $(X_v)_{v \in A}$, and given subsets $A, B, C \subseteq V$, with a slight abuse of notation we shall write $A \ci B \ |\ C$ to mean $X_A \ci X_B \ | \ X_C.$

Now, suppose that $A,B,C,$ and $D$ are disjoint collections of random variables. Then we may consider the following properties, known as the \emph{axioms of conditional independence}:

\bi
\item (P1) \emph{Symmetry:} $A \ci B \ | \ C$ implies $B \ci A \ | \ C$;
\item (P2) \emph{Decomposition:} $A \ci (B,C) \ | \ D$ implies $A \ci B \ | \ D$ and $A \ci C \ | \ D$;
\item (P3) \emph{Weak Union:} $A \ci (B,C) \ | \ D$ implies $A \ci B \ | \ (C,D)$;
\item (P4) \emph{Contraction:} $A \ci B \ | \ (C,D)$ and $A \ci C \ | \ D$ implies $A \ci (B,C) \ | \ D$.
\item (P5) \emph{Intersection:} $A \ci B \ | \ (C,D)$ and $A \ci C \ | \ (B,D)$ implies $A \ci (B,C) \ | \ D$.
\ei

We note that properties (P1) -- (P4) can be proved for arbitrary collections of random variables (see Lemma \ref{lem:p1p4} below). Property (P5) does not hold under such extreme generality, but the following result does hold (as verified in \cite{Lauritzen}, for example).

\begin{prop}[\cite{Lauritzen}] \label{prop:P1P5} 
Suppose for all $v \in V$ that $X_v$ is a random variable. If every finite subcollection of the $X_v$ has a joint density which is everywhere positive, then (P1) -- (P5) hold for all finite subsets $A, B, C, D \subseteq \{X_v\ | \ v \in V\}$.
\end{prop}

\begin{rek}
In definition \ref{def:p1p5} we used the * in the labeling of properties (P1*) -- (P5*) and $\perp$ instead of $\ci$ to distinguish those properties as abstract properties of a ternary relation. This is as opposed to (P1) -- (P5) above which are specifically related to conditional independence.
\end{rek}

\begin{rek}
The only significant difference between (P1)--(P5) and (P1*)--(P5*) occurs in the statement of (P5*). As we will show in Proposition \ref{prop:P5}, (P5) and (P5*) are equivalent when considering a finite collection of random variables. However, the more nuanced property (P5*) is necessary in order to obtain probability theoretic results for \emph{infinite} collections of random variables, which cannot be obtained from (P5) alone.
\end{rek}

\subsection{Markov Properties and Graphical Models}
Suppose now that $\mathcal{G} = (V,E)$ is an undirected graph with $|V| < \infty$. Let $\textbf{X} = (X_1,...,X_{|V|})$ be a $|V|$-variate random variable. Then we say that $\textbf{X}$ satisfies the (P) pairwise, (L) local, or (G) global Markov properties with respect to the graph $\mathcal{G}$ provided that:
\bi
\item (P) for each pair $(i,j)$ of vertices,
$$i \not \sim_{\mathcal{G}} j \textrm{ (i.e., }i\textrm{ and }j\textrm{ not adjacent)} \Rightarrow X_i \ci X_j \ | \ X_{V\setminus \{i,j\}};$$
\item (L) for each vertex $i$,
$$X_i \ci X_{V\setminus \cl(i)} \ | \ X_{\neigh(i)},$$
\item (G) for each triple $(A,B,S)$ of disjoint subsets of $V$ with $S$ separating $A$ and $B$, we have
$$X_A \ci X_B \ | \ X_S.$$
\ei

The following foundational result of Pearl and Paz \cite{PearlPaz} shows that, under general circumstances, these three Markov properties are equivalent.

\begin{thm}[Pearl and Paz, 1985 \cite{PearlPaz}]
Suppose $V$ is a finite set, and that $\textbf{X}$ is a collection of random variables indexed by $V$ such that for any disjoint non-empty subsets $A,B,C,D \subseteq V$ the intersection property
$$\textrm{(P5) } X_A \ci X_B \ | \ X_{C \cup D} \textrm{ and } X_A \ci X_C \ | \ x_{B \cup D} \Rightarrow X_A \ci X_{B \cup C} \ | \ X_D$$
holds. Then with respect to a given graph $\mathcal{G}$,
$$ \textbf{X} \textrm{ satisfies (P) } \Leftrightarrow \textbf{X} \textrm{ satisfies (L) } \Leftrightarrow \textbf{X} \textrm{ satisfies (G) }.$$
\end{thm}


\section{Graphoid Relations and Separation for Infinite Graphs} \label{sec:AGF}

We now introduce a graphical framework in which we can analyze the Markov properties from a purely set theoretic perspective which is independent of probability theory. Making this distinction in our mathematical treatment will allow us to understand better what drives the equivalence of the Markov properties. 

\subsection{Extended Graphoids and Semi-Graphoids}

Conditions (P1*) -- (P4*) are precisely the same as (P1) -- (P4) when the relation in question is that induced by conditional independence (as in Definition \ref{def:condIndep}). While the statements of (P5*) and (P5) are different, (P5*) is in fact a generalization of (P5) which is equivalent when $V$ is finite, a result stated formally in Proposition \ref{prop:P5} in the Supplemental section.

We now recall Definition \ref{defi:smp} of the set-theoretic analogues of the Markov properties for ternary relations, which we will ultimately consider for semi-graphoid and extended graphoid relations.

\begin{defi*}
If $V$ is any set, $\mathcal{G}$ is any undirected graph on $V$, and $\cdot \perp \cdot \ | \ \cdot$ is any ternary relation on $\mathcal{P}(V)$, we say that $(V,\perp)$ satisfies the pairwise (P*), local (L*), or global (G*) \emph{set-Markov properties} with respect to $\mathcal{G}$ provided that 
\bi
\item (P*) for each pair of vertices $\{i,j\} \subseteq V$, we have $i \not \sim_{\mathcal{G}} j \Rightarrow \{i\} \perp \{j\} \ | \ V \setminus \{i,j\};$
\item (L*) for each vertex $i$, we have $\{i\} \perp V \setminus \cl(i) \ | \ \neigh(i);$
\item (G*) for each triple (A,B,S) of disjoint subsets of $V$ with $S$ separating $A$ from $B$, we have $A \perp B \ | \ S.$
\ei

We say that (P*), (L*), or (G*) \emph{holds over} $\mathcal{G}$ when $(V,\perp)$ is clear from context.
\end{defi*}

Note that if $V$ is finite, and $\cdot \perp \cdot \ |\ \cdot$ is the conditional independence relation from Definition \ref{def:condIndep}, these set-Markov properties are precisely the usual probabilistic Markov properties. That is, (P) = (P*), (L) = (L*), and (G) = (G*). 

We now prove Theorem \ref{thm:CIREquiv}, which asserts the equivalence of the set-theoretic Markov properties under (P1*) -- (P5*).

\begin{thm*}
Suppose $V$ is any set, and $\cdot \perp \cdot \ | \ \cdot $ satisfies (P1*) -- (P4*). Then
$$(V,\perp) \textrm{ satisfies (G*) } \Rightarrow (V,\perp) \textrm{ satisfies (L*) }\Rightarrow (V,\perp) \textrm{ satisfies (P*)}.$$
If in addition $\cdot \perp \cdot \ | \ \cdot$ satisfies (P5*), then we also have
$$(V,\perp) \textrm{ satisfies (P*) } \Rightarrow (V,\perp) \textrm{ satisfies (G*)}.$$
\end{thm*}
\begin{proof}
\ \\
\noindent (G*) $\Rightarrow$ (L*):\\
For any subset $A \subseteq V$, the neighbor set ne$(A)$ separates $A$ and $V \setminus (A \cup \ne(A))$, so (G*) trivially implies (L*).
\ \\

\noindent (L*) $\Rightarrow$ (P*):\\
Let $i,j$ be any two vertices in $V$ which are not adjacent. Then $j \not \in \cl(i)$. By (L*), $i \perp V\setminus \cl(i) \ | \ \neigh(i)$. Since $j \not \in \cl(i)$, we may conclude from (P3*) that $i \perp j \ | \ V \setminus \{i,j\}$, proving (P*).
\ \\

\noindent (P*) $\Rightarrow$ (G*):\\
Suppose that $A,B,S \subseteq V$ are such that $S$ separates $A$ and $B$. Let $\widetilde{A}$ be the set of vertices in $V$ that can be reached by a path starting at some vertex in $A$ and which does not include any vertex in $S$. Define $\widetilde{B} = V \setminus (\widetilde{A} \cup S)$, noting that $S$ separates $\widetilde{A}$ and $\widetilde{B}$, and that $B \subseteq \widetilde{B}$.

Fix a vertex $i \in \widetilde{A}$. Then, since $i$ is separated from $\widetilde{B}$ by $S$, we know in particular that for any $j \in \widetilde{B}$, $i$ and $j$ are not adjacent. Therefore, by (P*), we have $ i \perp j \ | \ V \setminus \{i,j\}$, and we also have that $V \setminus \{i,j\} = (\widetilde{A}\setminus\{i\}) \cup S \cup (\widetilde{B}\setminus \{j\})$. Since this holds for all $j$ in $\widetilde{B}$, we have by property (P5*) that $i \perp \widetilde{B} \ | \ (\widetilde{A} \setminus i) \cup S,$ and a second application of (P5*) gives $\widetilde{A} \perp \widetilde{B} \ | \ S$.

Finally, since $A \subseteq \widetilde{A}$ and $B \subseteq \widetilde{B}$, two applications of property (P2*) allow us to conclude $A \perp B \ | \ S,$ finishing the proof that (P*) $\Rightarrow$ (G*).
\end{proof}

\begin{rek}
We note explicitly that Theorem \ref{thm:CIREquiv} is valid for $V$ of arbitrary (even uncountable) cardinality. Moreover, although Theorem \ref{thm:CIREquiv} deals with $V$ of arbitrary cardinality, it is not necessary to invoke the axiom of choice to obtain the result.
\end{rek}

We are now in a position to deduce the result of Pearl and Paz \cite{PearlPaz} formulated in terms of set-Markov properties as a special case of Theorem \ref{thm:CIREquiv} above.

\begin{cor}[Pearl and Paz, 1985 \cite{PearlPaz}]\label{thm:PearlPaz}
Suppose $V$ is a finite set and $\cdot \perp \cdot \ | \ \cdot$ is a ternary relation on $\mathcal{P}(V)$ satisfying (P1*)--(P4*), and is such that for any disjoint non-empty subsets $A,B,C,D \subseteq V$ the intersection property
$$\textrm{(P5) } A \perp B \ | \ C \cup D \textrm{ and } A \perp C \ | \ B \cup D \Rightarrow A \perp B \cup C \ | \ D$$
holds. Then
$$ \cdot \perp \cdot \ | \ \cdot \textrm{ satisfies (P) } \Leftrightarrow \cdot \perp \cdot \ | \ \cdot \textrm{ satisfies (L) } \Leftrightarrow \cdot \perp \cdot \ | \ \cdot \textrm{ satisfies (G) }.$$
\end{cor}
\begin{proof}
By Proposition \ref{prop:P5}, the assumption of (P5) is equivalent to that of (P5*). Thus, we may apply Theorem \ref{thm:CIREquiv}.
\end{proof}

\begin{rek} The only important difference between the proof of Theorem \ref{thm:CIREquiv} and the proof of Corollary \ref{thm:PearlPaz} from \cite{PearlPaz} is in the proof of (P*) $\Rightarrow$ (G*). For this step, \cite{PearlPaz} relies on the pairwise statement of (P5) and a reverse induction argument starting with the largest separator for which the claim fails. Our proof instead appeals to the specific way in which (P5*) is formulated. Combining the equivalence of (P5) and (P5*) from Proposition \ref{prop:P5} with the argument we used to prove Theorem \ref{thm:CIREquiv} would provide an alternate proof of the equivalence of the usual finite, probabilistic Markov properties.
\end{rek}

One could ask the question of whether the general intersection property (P5*) is the minimal assumption required to get equivalence of the set-Markov properties. We will now show that under (P1*) -- (P4*), the assumption (P5*) is equivalent to the statement (P*) $\Rightarrow$ (G*). 

\begin{prop} \label{prop:converse}
If (P1*)--(P4*) hold, and (P*) $\Rightarrow$ (G*), then (P5*) must hold. 
\end{prop}

\begin{proof}
Suppose that (P*) implies (G*), and that for some $I \subseteq V$, and all $i \in I$, we have 
$$X_i \perp Y \ | \ (X_j)_{j \in I \setminus \{i\}}, Z.$$
Then let $\mathcal{G}$ be the graph on $V$ with one node corresponding to each of the $X_i$, and also nodes for $Y$ and $Z$, and let the edge set be chosen so that there is an edge between two nodes precisely if the (P*) condition for those two nodes holds. That is, if $a$ and $b$ are two nodes, then $(a,b) \in E$ if and only if $a \perp b \ | \ V \setminus \{a,b\}$. Then for this choice of $E$, (P*) trivially holds, and since we are assuming (P*) implies (G*), we also have that (G*) holds. By hypothesis, there are no edges from any $X_i$ to $Y$, and so $Z$ separates $X_I$ and $Y$. Therefore by (G*) we may conclude $X_I \perp Y \ | \ Z$. 
\end{proof}

The above proposition shows that (P5*) is a necessary property to obtain the equivalence of the set-Markov properties, and so is indeed minimal. Importantly, this affirms the notion that (P5*) is the correct generalization of (P5).

\section{Markov Properties for Infinite Sets of Random Variables}\label{sec:CI}

Our goal in this section is to obtain general conditions under which (P5*) holds for the conditional independence relation. This will allow us to verify the equivalence of the infinite probabilistic Markov properties. 

\subsection{Axioms of Conditional Probability}

The first step toward this goal requires us to introduce the most general definition of conditional independence, significantly generalizing Definition (\ref{def:condIndep}).

\begin{defi} \label{def:condIndepSA}
Given collections of random variables $A$, $B$, and $C$, we say that
$$A \ci B \ | \ C$$
provided the $\sigma$-algebras $\sigma(A)$, $\sigma(B)$, and $\sigma(C)$ satisfy
\be \mathbb{P}(E_1|\sigma(C))\mathbb{P}(E_2|\sigma(C)) = \mathbb{P}(E_1 \cap E_2|\sigma(C)) \label{ln:indep} \ee
for all events $E_1 \in \sigma(A)$ and $E_2 \in \sigma(B)$. 
\end{defi}

It turns out that properties (P1*)--(P4*) hold for any collection of random variables, even for this more general form of the conditional independence relation. 

\begin{lem}\label{lem:p1p4}
If $(X_v)_{v \in V}$ is any collection of random variables with $V$ at most countably infinite, then (P1*)--(P4*) hold with respect to the conditional independence relation $\cdot \ci \cdot \ | \ \cdot$. 
\end{lem}
\begin{proof}
The proof follows from measure theoretic arguments. Similar results are commonly referenced in the literature, and some proofs can be found, for example, in \cite{Dawid, meyer1966probability}. However, we have been unable to find a reference that lists each of the properties (P1*)--(P4*) at our level of generality. Thus, we include our own proofs in the Supplemental section for the sake of completeness.
\end{proof}

In the next section, we will consider the generality in which (P5*) holds.

\subsection{Sufficient Conditions for the General Intersection Property}

When the vertex set $V$ is finite, there are well-known sufficient conditions for the intersection property (P5) to hold (see \cite{Peters} for an approach involving densities and product measures, and see \cite{Dawid} for a more general $\sigma$-algebra oriented approach). In particular, (P5) holds if the random variables have a joint density which is positive everywhere. However, in the infinite setting, sufficient conditions for the general intersection property (P5*) to hold are not immediately evident. We address this by introducing a set of sufficient conditions under which the general intersection property holds for the conditional independence relation.

\begin{defi}\label{def:d1d2}
Let $\{X_i \ | \ i \in \mathbb{N}\}$ be a collection of real-valued random variables. The $X_i$ are said to satisfy the \emph{Infinite Intersection Property} (IIP) and \emph{Decorrelation Property} (DCP) respectively provided that:\\

(IIP) Given any (possibly infinite) subset $D \subseteq \mathbb{N}$, and any finite $A, B, C \subseteq \mathbb{N}\setminus D$, we have that $$(X_A \ci X_B \ | \ X_C, X_D \textrm{ and } X_A \ci X_C \ | \ X_B, X_D)  \Rightarrow X_A \ci (X_B,X_C) \ | \ X_D.$$ \ 

(DCP) Given any (potentially infinite) subset $D \subseteq \mathbb{N}$, and any event $$E \in \cap_{n} \sigma(X_D,X_n,X_{n+1},X_{n+2},...),$$ there exists an event $E' \in \sigma(X_D)$ such that $\mathbb{P}(E \Delta E') = 0$.

\end{defi}

\begin{rek} \label{rek:d2}
In the finite-dimensional setting, the infinite intersection property (IIP) reduces precisely to (P5), since in that case all subsets $A,B,C,$ and $D$ under consideration will be finite. In addition, the decorrelation property (DCP) holds trivially in the finite-dimensional setting since, in that case,
$$\cap_n \sigma(X_D, X_n, X_{n+1},...) = \sigma(X_D).$$
Thus, the assumption of both the infinite intersection property (IIP) and decorrelation property (DCP) reduces precisely to (P5) in the finite-dimensional setting. As we shall see, (IIP) and (DCP) are also sufficient to obtain (P5*) in the infinite-dimensional setting.
\end{rek}

\begin{rek} \label{rek:dp}
One can think of (IIP) as an intermediate infinite extension of (P5) which is necessary for, but not quite as strong as (P5*), and hence is more readily verifiable. (DCP) is not a necessary condition for (P5*), but provides a quantitative formulation of the idea that the infinite collection of variables is sufficiently ``decorrelated,'' which in turn allows the verification of (P5*). 
\end{rek}

\begin{rek}
It may appear that verifying (DCP) for $D = \emptyset$ implies (DCP) for all other possible choices of $D$, with the key step in such an argument being
$$\cap_n \sigma(X_D,X_n,X_{n+1},...) = \sigma(X_D, \cap_n \sigma(X_n,X_{n+1},...)).$$
However, the above expression does not hold in general, even if equality is weakened to equivalence, in the sense that two $\sigma$-algebras $\mathcal{F}$ and $\mathcal{G}$ are equivalent if for any $A \in \mathcal{F}$ there exists $B \in \mathcal{G}$ with $\mathbb{P}(A \Delta B) = 0$, and vice versa (some conditions under which such an equivalence holds are listed in \cite{Weizsacker}). Because this equality does not always hold, it is indeed necessary to verify (DCP) for all subsets $D \subseteq \mathbb{N}$.
\end{rek}

\begin{thm} \label{thm:d1d2thm}
Suppose that $\{X_n \ | \ n \in \mathbb{N}\}$ is a collection of random variables satisfying properties (IIP) and (DCP). Then the ternary relation $\cdot \ci_X \cdot \ | \ \cdot$ corresponding to conditional independence satisfies (P5*). 
\end{thm}

\begin{proof}
Let $I, D \subseteq \mathbb{N}$ be arbitrary (potentially infinite) subsets, let $C \subseteq \mathbb{N}$ be finite, and suppose that $X_i \ci X_C \ | \ (X_j)_{j \in I \setminus \{i\}}, X_D$ for all $i \in I$. By iterative application of property (IIP) we may conclude that for any finite $J \subseteq I$,
$$X_J \ci X_C \ | \ X_{I \setminus J}, X_D.$$
We may therefore assume that $I$ is infinite, since the above verifies (P5*) when $I$ is finite. 

Let $I = \{i_1,i_2,i_3,...\}$, and let $I_n = \{i_n,i_{n+1},i_{n+2},...\}$. Applying property (P2) gives 
$$X_J \ci X_C \ | \ X_{I \setminus K}, X_D$$
for any $K$ satisfying $J \subseteq K \subseteq I$, and so for all sufficiently large $n$, we have 
$$X_J \ci X_C \ | \ X_{I_n}, X_D.$$ 

Now, let $A$ be an event in $\sigma(X_J)$ and $B$ an event in $\sigma(X_C)$. Then by the backward martingale convergence theorem (for a reference, see e.g.\ \cite{Klenke}), it holds pointwise that
$$\lim_{n\rightarrow \infty} \mathbb{P}(A | X_{I_n},X_D) = \mathbb{P}(A | \mathcal{F}),$$
$$\lim_{n\rightarrow \infty} \mathbb{P}(B | X_{I_n},X_D) = \mathbb{P}(B | \mathcal{F}),\textrm{ and}$$
$$\lim_{n\rightarrow \infty} \mathbb{P}(A \cap B | X_{I_n},X_D) = \mathbb{P}(A \cap B | \mathcal{F}),$$
where $\mathcal{F} = \cap_n \sigma(X_{I_n},X_D)$. Thus, since the independence factorization holds almost surely for all sufficiently large $n$, we have that almost surely $\mathbb{P}(A|\mathcal{F})\mathbb{P}(B|\mathcal{F}) = \mathbb{P}(A \cap B|\mathcal{F})$.

Now, note that $\sigma(X_D) \subseteq \mathcal{F}$, and by property (DCP), for any $E \in \mathcal{F}$, there exists $E' \in \sigma(X_D)$ satisfying $\mathbb{P}(E \Delta E') = 0$. We may therefore apply Lemma \ref{lem:CEE}, and obtain that $\mathbb{P}(A|\mathcal{F}) = \mathbb{P}(A|X_D)$, and similarly for $B$ and $A \cap B$, and so we have $A \ci B \ | \ X_D$.

Since $A$ and $B$ were arbitrary elements of $\sigma(X_J)$ and $\sigma(X_C)$ respectively, we have that $X_J \ci X_C \ | \ X_D$ for all finite $J \subseteq I$. 

Now, let $E \in \sigma(X_I)$. By Proposition \ref{prop:algapprox}, there is a sequence of sets $E_n \in \sigma(X_{i_1},...,X_{i_n})$ such that $\mathbb{P}(E \Delta E_n) \rightarrow 0$. Since $E_n \in \sigma(X_J)$ for a finite $J \subseteq I$, we have $E_n \ci X_C \ | \ X_D$. Let $F \in \sigma(X_C)$ be arbitrary. Then
\be \mathbb{P}(E_n|X_D)\mathbb{P}(F|X_D)-\mathbb{P}(E_n\cap F|X_D) = 0 \label{ln:lims} \ee
for all $n$ almost surely. Since $\mathbb{P}(E \Delta E_n) \rightarrow 0$, we have that $\mathbb{E}[1_{E\Delta E_n}] \rightarrow 0$. So $\mathbb{E}[\mathbb{E}[1_{E\Delta E_n} | X_D]] \rightarrow 0$, and since $\mathbb{E}[1_{E \Delta E_n} | X_D] = \mathbb{P}(E\Delta E_n | X_D) \geq 0$, we may conclude that $\mathbb{P}(E\Delta E_n | X_D) \rightarrow 0$ almost surely. Thus, since 
$$ |\mathbb{P}(E| X_D) - \mathbb{P}(E_n | X_D)| \leq |\mathbb{P}(E \setminus E_n| X_D)| + |\mathbb{P}(E_n \setminus E | X_D)| = |\mathbb{P}(E \Delta E_n | X_D)|,$$
we may conclude that $\mathbb{P}(E_n|X_D) \rightarrow \mathbb{P}(E|X_D)$ almost surely. A similar argument works to show that $\mathbb{P}(E_n \cap F|X_D) \rightarrow \mathbb{P}(E \cap F|X_D)$ almost surely, and therefore we may obtain from line \eqref{ln:lims} that almost surely,
$$\mathbb{P}(E|X_D)\mathbb{P}(F|X_D)-\mathbb{P}(E\cap F|X_D) = 0$$
This finishes the proof that $X_I \ci X_C \ | \ X_D$ for finite $C$.

Suppose now that $C$ is infinite. Then for any finite subset $C' \subseteq C$, we have $X_i \ci X_{C'} \ | \ (X_j)_{j \in I \setminus\{i\}}, X_D$, and so the above arguments show that $X_I \ci X_{C'} \ | \ X_D$ for any finite $C' \subseteq C$. Note that $\mathcal{C}:= \{ E \in \sigma(X_C) \ | \ E \in \sigma(X_{C'}) \text{ for some finite } C' \subseteq C\}$ is a $\pi$-system, and that for any event $E_I \in \sigma(X_I)$, the set of events $E \in \sigma(X_C)$ satisfying 
\be \mathbb{P}(E_I | X_D) \mathbb{P}(E | X_D) = \mathbb{P}(E_I \cap E \ | \ X_D) \label{ln:dpl} \ee
forms a $\lambda$-system containing $\mathcal{C}$. We may then conclude by Dynkin's $\pi$-$\lambda$ theorem that line \eqref{ln:dpl} is satisfied for all $E_I \in \sigma(X_I)$ and all $E \in \sigma(\mathcal{C}) = \sigma(X_C)$, and so $X_I \ci X_C \ | \ X_D$ even if $C$ is infinite. Thus we have demonstrated (P5*).
\end{proof}

The following result provides a convenient way to verify (DCP) in practice.

\begin{lem} \label{lem:d2new}
Let $(X_n)_{n \in \mathbb{N}}$ be a collection of random variables satisfying the following: \\
For any $D \subseteq \mathbb{N}$, let $\mathcal{F}_{D,-n} = \sigma(X_D,X_n,X_{n+1},...)$. For any $m \in \mathbb{N}$, and $\mathbb{P}$-a.e.\ value of $X_D = x_D$, there is a function $g_{m,D,x_D}(n)$ satisfying
\ben 
\item  $\lim_{n \rightarrow \infty}g_{m,D,x_D}(n) = 0$, and
\item  For any $A \in \mathcal{F}_{D,-n}$ we have
\be \var(\mathbb{P}(A|X_1,...,X_m,X_D)|X_D=x_D) \leq g_{m,D,x_D}(n) \textrm{ a.s.}. \label{ln:varineq} \ee
\een
Then the $(X_n)_{n \in \mathbb{N}}$ satisfy property (DCP).

More strongly, for (DCP) to hold, it is sufficient that inequality \eqref{ln:varineq} be valid for all $A \in \mathcal{F}_{D,-n,-N}$ (for all  $N > n$), where
$$\mathcal{F}_{D,-n,-N}:= \sigma(X_D,X_n,X_{n+1},...,X_N).$$
\end{lem}

\begin{proof}
Let $D \subseteq \mathbb{N}$ and$A \in \cap_n \mathcal{F}_{D,-n}$ be arbitrary. By assumption,
$$\var(\mathbb{P}(A|X_1,...,X_m,X_D)|X_D=x_D) \leq g_{m,D,x_D}(n) \textrm{ a.s. }X_D \ \ \forall n,$$
and thus
$$\var(\mathbb{P}(A|X_1,...,X_m,X_D)|X_D=x_D) = 0 \textrm{ a.s. }X_D.$$
Thus, we have that $\mathbb{P}(A|X_1,...,X_m,X_D)$ is almost surely just a function of $X_D$. That is, $$\mathbb{P}(A|X_1,...,X_m,X_D) = \mathbb{P}(A|X_D)$$ almost surely for all $m$. By the L\'e{}vy zero-one law, we may conclude 
$$1_A = \lim_{m \rightarrow \infty} \mathbb{P}(A|X_D,X_1,...,X_m) = \mathbb{P}(A|X_D)$$
almost surely, so there is some $A' \in \sigma(X_D)$ such that $1_A = 1_{A'}$ almost surely, or equivalently, $\mathbb{P}(A \Delta A') = 0$. Thus (DCP) holds.

Now, suppose instead that we only assume
$$\var(\mathbb{P}(E|X_1,...,X_m,X_D)|X_D=x_D) \leq g_{m,D,x_D}(n) \textrm{ a.s. }X_D$$
for those $E$ contained in $\mathcal{F}_{D,-n,-N} = \sigma(X_D,X_n,X_{n+1},...,X_N)$ for some $n, N$. The algebra given by $\mathcal{A} = \cup_N \sigma(X_D,X_n,X_{n+1},...,X_N)$ generates $\mathcal{F}_{D,-n}$, and so by Proposition \ref{prop:algapprox}, for all $n \in \mathbb{N}, \epsilon > 0$, there exists some $N \in \mathbb{N}$ and an event $A_{n,\epsilon} \in \mathcal{F}_{D,-n,-N}$ such that $\mathbb{P}(A \Delta A_{n,\epsilon}) < \epsilon$. Thus, in particular, we have $1_{A_{n,\epsilon}} \rightarrow 1_A$ almost surely. 

By the dominated convergence theorem, we may thus conclude
$$\var(\mathbb{P}(A_{n,\epsilon}|X_1,...,X_m,X_D)|X_D=x_D) \rightarrow \var(\mathbb{P}(A|X_1,...,X_m,X_D)|X_D=x_D),$$
and therefore that
$$\var(\mathbb{P}(A|X_1,...,X_m,X_D)|X_D=x_D) \leq g_{m,D,x_D}(n)$$
for all $n$ (since $A \in \cap_n \mathcal{F}_{D,-n}$). Then we may follow the above argument to obtain the existence of $A' \in \sigma(X_D)$ such that $\mathbb{P}(A \Delta A') = 0$, and conclude that (DCP) holds.
\end{proof}

\subsection{Analysis of Properties (IIP), (DCP), and (P5*)}
There are a few natural questions that immediately arise regarding the technical conditions required to obtain the equivalence of the Markov properties. In this section, we consider the following three:

\begin{itemize}
\item Is a more straightforward extension of the intersection property (as given by (P5) in the finite setting) sufficient to obtain the equivalence of the various Markov properties in the infinite setting?
\item Does the infinite intersection property (IIP) imply the decorrelation property (DCP) or vice versa?
\item Is the infinite intersection property (IIP) or the decorrelation property (DCP) in any sense a necessary condition for (P5*)?
\end{itemize}

We begin by addressing the first question. One may consider the following na{\"i}ve extension of (P5) to the infinite case: assume that (P5) holds for all finite subcollections of the infinite set of variables. This assumption is perhaps the most natural way of extending (P5), and is strictly weaker than the infinite intersection property (IIP) (which itself can be viewed as an extension of (P5)). However, this assumption is not sufficient to obtain the equivalence of the Markov properties, as shown in Example \ref{exa:5.1} in the Supplemental Section. In fact, we show in Example \ref{exa:5.1} that even adding the decorrelation property (DCP) to this na{\"i}ve extension of (P5) is not sufficient to obtain the equivalence of the Markov properties. This demonstrates the ineffectiveness of the na{\"i}ve extension of (P5). The infinite intersection property (IIP) augments this extension of (P5) only by allowing the conditioning set to be infinite, and in light of Theorem \ref{thm:d1d2thm} can be viewed as the correct extension of property (P5).

We now consider the question of whether (DCP) implies (IIP) or vice versa. First note that, as mentioned in Remark \ref{rek:d2} above, in the finite setting the decorrelation property (DCP) holds trivially, and the infinite intersection property (IIP) is equivalent to the usual intersection property (P5). Since (P5) does not hold in general even in the finite case, this shows that the decorrelation property (DCP) alone is not a sufficient condition for (P5*), and also that (DCP) does not imply (IIP). The previously mentioned Example \ref{exa:5.1} also provides a less trivial example where the decorrelation condition (DCP) is satisfied and where (P5*) is not. This shows that the decorrelation property (DCP) by itself does not imply the infinite intersection property (IIP) nor the general intersection property (P5*).

Next, we address whether (IIP) implies (DCP). In fact, (IIP) does not imply (DCP), and moreover (IIP) is not a sufficient condition for (P5*) by itself. We demonstrate this in Example \ref{exa:5.2} in the Supplemental Section, in which we provide a collection of random variables which satisfies (IIP), but which does not satisfy (DCP), and for which the equivalence of the Markov properties does not hold. This shows that (IIP) by itself is not enough to imply (P5*), and thus demonstrates the importance of also verifying (DCP) when one wishes to appeal to the equivalence of the Markov properties for an infinite collection of random variables.

The question that remains is whether the properties (IIP) and/or (DCP) are necessary for (P5*). It is easy to see that (IIP) is directly implied by (P5*), so that (IIP) is indeed a necessary condition for the equivalence of the Markov properties. However, the theoretical relationship between (P5*) and (DCP) is less clear. 

As demonstrated by Theorem \ref{thm:d1d2thm}, (DCP) is a convenient criterion which can be used with (IIP) to verify the equivalence of the Markov properties. In addition, for many applications of interest, (DCP) is satisfied due to the ``decorrelation'' of the random variables (see Remark \ref{rek:dp}). However, we demonstrate in Example \ref{exa:5.3} of the Supplemental Section that (DCP) is not a necessary condition for (P5*). Despite this, the following proposition gives a partial converse to Theorem \ref{thm:d1d2thm}. In particular, this proposition provides some settings under which some of the implications required by the definition of (DCP) are implied by (P5*), thus providing more justification that (DCP) is a reasonable condition to impose.

\begin{prop} \label{prop:d2conv}
Suppose that $(X_n)_{n=1}^\infty$ is a collection of random variables for which (P5*) holds, and that $D\subseteq \N$ and $H = \{h_1, h_2, ...\} \subseteq \mathbb{N}$ are such that in the graph on $\N$ induced by (P*), there are no paths between distinct elements of $H$ that do not pass through $D$. Then for any event $E \in \cap_n \sigma(X_D, X_{h_n}, X_{h_{n+1}},...)$, there exists $E' \in \sigma(X_D)$ such that $\P(E \Delta E') = 0$. 
\end{prop}

\begin{proof}
Since (P5*) holds, (P*) is equivalent to (G*), and so the condition on the elements of $H$ relative to the (P*)-induced graph implies that for any $h_i, h_j \in H$ we have $X_{h_i} \ci X_{h_j} \ | \ X_D$.

Now, suppose that $E \in \cap_n \sigma(X_D, X_{h_n}, X_{h_{n+1}},...)$. Then $E \in \sigma(X_D, X_{h_1}, X_{h_2},...)$, so it is the case that $1_E = \lim_{m \to \infty} \P(E \ | \ X_D, X_{h_1},...,X_{h_m})$. But, by the conditional independences $X_{h_i} \ci X_{h_j} \ | \ X_D$, it is the case that $(X_{h_1},...,X_{h_n}) \ci (X_{h_{n+1}}, X_{h_{n+2}},...) \ | \ X_D$, and so since $E \in \sigma(X_{h_{n+1}}, ...)$, we have for all $n$ the equality $\P(E \ | \ X_D, X_{h_1}, X_{h_2},...,X_{h_n}) = \P(E \ | \ X_D)$ with probability one. Thus, with probability one, we have $1_E = \P(E \ | \ X_D)$. Thus, if $E' = \{\P(E \ | \ X_D) = 1\}$, then $E' \in \sigma(X_D)$ and $\P(E \Delta E') = 0$, concluding the proof.
\end{proof}

\begin{rek}
The infinite star graph is an example of the kind of graph to which Proposition \ref{prop:d2conv} could be applied. For this example, we could take $D$ to be the hub node, and $H$ to contain all of the remaining nodes. 
\end{rek}

\section{Graphical Gaussian Processes} \label{sec:GGP}

Two of the most important classes of graphical models that have been studied in the literature are Gaussian graphical models and discrete log-linear models. In this section and the one that follows, we give a comprehensive treatment of their countably infinite analogues and important special cases thereof. 

\subsection{Infinite-Dimensional Gaussian Graphical Models and Properties}

We begin by introducing some definitions and notation.

\begin{defi}\label{def:gprocs}\ 
\bi
\item A \emph{Gaussian process} is a collection of random variables $(X_i)_{i \in I}$ such that each $X_i$ is a normal random variable, and every finite collection of the $X_i$ has a multivariate normal distribution. We will assume in this paper that in a Gaussian process, all of the random variables have mean zero. For any subset $A \subseteq I$, we use $X_A$ to denote the random variable $(X_a)_{a \in A}$. 
\item For $A, B \subseteq I$, we use $\Sigma_{AB}$ to represent the cross-covariance matrix of $X_A$ and $X_B$. We also write $\Sigma_A$ to represent $\Sigma_{AA}$.
\item If $A, B \subseteq I$ are finite and disjoint, we use $\Sigma_{A|B}$ to represent the covariance matrix of $X_A$ given $X_B$. It is well known that when $\Sigma_B$ is invertible,
$$\Sigma_{A|B} = \Sigma_A - \Sigma_{AB}\Sigma_B^{-1}\Sigma_{BA}.$$
\item For a finite-dimensional matrix $M$, we use the notation $\lambda_i(M)$ to represent the $i^{\text{th}}$ largest eigenvalue of $M$.
\ei
\end{defi}

We now introduce a theorem which is a critical ingredient in the verification of (IIP) and (DCP) for Gaussian processes.

\begin{thm} \label{thm:density}
Let $(X_n)_{n \in \mathbb{N}}$ be a Gaussian process which satisfies bounds (1) and (2) given in the statement of Theorem \ref{thm:gaussianMR}. Then for any finite $A \subseteq \mathbb{N}$ of size $r$ and any (possibly infinite) $B \subseteq \mathbb{N}$ disjoint from $A$, there is a $\sigma(X_B)$-measurable function $\mu_{A|B}$ and an $r \times r$ symmetric, positive-definite matrix $\Sigma_{A|B}$ such that
$$X_A | X_B \sim \mathcal{N}(\mu_{A|B}(X_B), \Sigma_{A|B}).$$

\end{thm}
\begin{proof}
This theorem is central to ensuing arguments, but due to its technical nature and length, the proof is provided in Appendix \ref{app:MR}.
\end{proof}

\begin{lem} \label{lem:d1}
Under the assumptions of Theorem \ref{thm:gaussianMR}, the Gaussian process $(X_n)_{n \in \mathbb{N}}$ satisfies (IIP).
\end{lem}
\begin{proof}
The proof follows by exploiting the existence of a density using standard techniques. Details are provided in the Supplemental section.
\end{proof}

\begin{rek}
The proof of Lemma \ref{lem:d1} amounts to verifying an only slightly extended version of the finite-dimensional intersection property (P5), since the only difference between (P5) and (IIP) lies in potentially conditioning on an infinite-dimensional random variable. For more on the conditions under which the intersection property holds in the finite-dimensional case, see \cite{Peters}.
\end{rek}

\begin{lem} \label{lem:d2}
Under the assumptions of Theorem \ref{thm:gaussianMR}, the Gaussian process $(X_n)_{n \in \mathbb{N}}$ satisfies the decorrelation property (DCP).
\end{lem}
\begin{proof}
We will appeal to Lemma \ref{lem:d2new}. Let $D \subseteq \mathbb{N}$, $X_D=x_D$, and $m \in \mathbb{N}$ be arbitrary. We need to show that there is a function $g_{m,D,x_D}(n)$ with $\lim_{n \rightarrow \infty}g_{m,D,x_D}(n) = 0$ which satisfies the property that for any $E \in \mathcal{F}_{D,-n,-N} := \sigma(X_D,X_n,X_{n+1},...,X_N)$, we have
\be \var(\mathbb{P}(E|X_1,...,X_m,X_D)|X_D=x_D) < g_{m,D,x_D}(n) \textrm{ a.s. }X_D. \nonumber \ee

By Theorem \ref{thm:density}, after conditioning on $X_D=x_D$ we know that $(X_1,...,X_m,X_n,...,X_N)$ has some conditional multivariate normal distribution. Let $E \in \mathcal{F}_{D,-n,-N}$ (with $n > m$). Then we have that

\bea & & \var(\mathbb{P}(E|X_1,...,X_m,X_D)|X_D=x_D) \nonumber \\
& =    & \mathbb{E}\left[\left(\mathbb{P}(E|X_1,...,X_m,X_D)\right.\right.\nonumber \\
&      & \left.\left.-\ \mathbb{E}[\mathbb{P}(E|X_1,...,X_m,X_D)|X_D=x_D]\right)^2|X_D=x_D\right] \nonumber \\
& \leq & \mathbb{E}\left[\left(\mathbb{P}(E|X_1=x_1,...,X_m=x_m,X_D)\right.\right. \nonumber \\
&      & \left.\left.-\ \mathbb{P}(E|X_1=0,...,X_m=0,X_D)\right)^2|X_D=x_D\right] \nonumber \\
& = & \mathbb{E}\left[\left(\int_0^1 \frac{d}{dt}|_{t=y} \mathbb{P}(E|(X_1,...,X_m)=t(x_1,...,x_m),X_D=x_D)dy\right)^2|X_D=x_D\right]. \nonumber \\
& &  \label{ln:varbnd} \eea
Consider the quantity
$$\frac{d}{dt}|_{t=y} \mathbb{P}(E|(X_1,...,X_m)=t(x_1,...,x_m),X_D=x_D).$$
If we let $f_{x_D}$ denote the corresponding multivariate normal density, we have
\bea 
& & \frac{d}{dt}|_{t=y} \mathbb{P}(E|(X_1,...,X_m)=t(x_1,...,x_m),X_D=x_D) \nonumber \\
& = & \frac{d}{dt}|_{t=y} \int_E f_{x_D}(x_n,...,x_N,tx_1,...,tx_m)dx_n...dx_N \nonumber \\
& = & \int_E \frac{d}{dt}|_{t=y} f_{x_D}(x_n,...,x_N,tx_1,...,tx_m)dx_n...dx_N \nonumber \\
& = & \int_E f_{x_D}(x_n,...,x_N,yx_1,...,yx_m)\frac{d}{dt}|_{t=y}\left(-\frac{1}{2}\left (v_{x_D}+t\left (\sum_{j=1}^m c_{n,j} x_j,...,\sum_{j=1}^m c_{N,j} x_j\right )\right )^T\right. \nonumber \\
& & \left.\Sigma_{\{n,...,N\}|\{1,...,m\}\cup D}^{-1}\left (v_{x_D}+t\left (\sum_{j=1}^m c_{n,j} x_j,...,\sum_{j=1}^m c_{N,j} x_j\right )\right )\right)  dx_n...dx_N, \label{ln:lastline} \eea
where $v_{x_D}$ is a $N-n+1$ dimensional vector depending on $x_D, n$, and $N$, which represents the total contribution of $X_D$ to the mean of $(X_n,...,X_N)$ given $X_D=x_D$, and $c_{i,j}$ is the factor by which $X_j$ contributes to the mean of $X_i$. More explicitly, we have
$$v_{x_D} = (\Sigma_{rD}\Sigma_{DD}^{-1}x_D)_{r=n}^N,$$
with $\Sigma_{rD}\Sigma_{DD}^{-1}$ defined as a whole by the same definition as in the proof of Theorem \ref{thm:density}, i.e.
\be (v_{x_D})_r = \sum_{d_1 \in D} \sigma_{rd_1} \sum_{d_2 \in B} \sigma^{d_1d_2}x_{d_2}, \label{ln:dubsum} \ee
and
\be c_{i,j} = \sum_k \sigma_{ik} \sigma^{kj} = O\left (\sum_k g_0(i,k) g_1(k,j)\right ) = O(g_2(i,j)). \label{ln:cij}  \ee
Note that the double sum from line \eqref{ln:dubsum} is convergent since $\sigma_{rd_1} = O(g_0(r,d_1))$, and 
\be |\sigma^{d_1d_2}| \leq |\sigma_{d_1d_2}| + \sum_k |\sigma_{d_1k}\sigma_{kd_2}| = O(g_1(d_1,d_2)) \nonumber \ee
and so
\bea (v_{x_D})_r & = & O\left (\sum_{d_1 \in D} \sum_{d_2 \in D} g_0(r,d_1)g_1(d_1,d_2) x_{d_2}\right ) \nonumber \\
& = & O\left (\sum_{d_2 \in B} g_2(r,d_2) x_{d_2}\right ) \label{ln:vxb} \\
& = & O_{x_D,r}(1)\textrm{ a.s.\ }, X_D \label{ln:normbndd} \eea
where here and elsewhere the subscript on the $O$ means that the implicit constant may depend on the subscripted variables, and we have obtained line \eqref{ln:normbndd} by noting that $x_{d_2} = O(d_2^{\epsilon})$ a.s.\ $X_D$ from Proposition \ref{prop:normsumbnd}.

The presence of the density $f_{x_D}$ in the expression from line \eqref{ln:lastline} gives
\bea & = & \mathbb{E}\Bigg[1_E \frac{d}{dt}\bigg|_{t=y}\left(-\frac{1}{2}\left(v_{x_D}+t\left(\sum_{j=1}^m c_{n,j} x_j,...,\sum_{j=1}^m c_{N,j} x_j\right)\right)^T\right. \nonumber \\
& & \left.\Sigma_{\{n,...,N\}|\{1,...,m\}\cup D}^{-1}\left(v_{x_D}+t\left(\sum_{j=1}^m c_{n,j} x_j,...,\sum_{j=1}^m c_{N,j} x_j\right)\right)\right)\nonumber \\
& & \Bigg | (X_1,...,X_m) = y(x_1,...,x_m),X_D=x_D \Bigg]. \label{ln:absvalquant} \eea

Letting $(\Sigma_{\{n,...,N\}|\{1,...,m\}\cup D}^{-1})_{ij} = \sigma^{ij}$, the quantity in the derivative above is equal to
\bea & & -\frac{1}{2} \sum_{k=n}^N \sum_{\ell=n}^N \sigma^{k\ell}\left(v_{x_D,k}+t\sum_{j=1}^m c_{k,j} x_j\right )\left (v_{x_D,\ell}+t\sum_{j=1}^m c_{\ell,j} x_j\right ) \nonumber \\
& = & -\frac{1}{2} \sum_{k=n}^N \sum_{\ell=n}^N \sigma^{k\ell}\left(t\left (v_{x_D,k}\sum_{j=1}^m c_{\ell,j} x_j+ v_{x_D,\ell}\sum_{j=1}^m c_{k,j} x_j\right )\right. \nonumber \\
&   & +\ \left. t^2\left (\sum_{j=1}^m c_{k,j} x_j\right )\left (\sum_{j=1}^m c_{\ell,j} x_j\right ) + \widetilde{C}\right) \nonumber \eea
where $\widetilde{C}$ represents a constant term not depending on $t$, and therefore we have that the derivative with respect to $t$ at $t=y$ is given by
\be - \frac{1}{2} \sum_{k=n}^N \sum_{\ell=n}^N \sigma^{k\ell}\left(v_{x_D,k}\sum_{j=1}^m c_{\ell,j} x_j + v_{x_D,\ell}\sum_{j=1}^m c_{k,j} x_j + 2y\left (\sum_{j=1}^m c_{k,j} x_j\right )\left (\sum_{j=1}^m c_{\ell,j} x_j\right )\right). \nonumber \ee

Thus, incorporating the results of lines \eqref{ln:cij} and \eqref{ln:vxb}, the derivative from line \eqref{ln:absvalquant} is bounded in absolute value by 
\bea & & O \left( \sum_{k=n}^N \sum_{\ell=n}^N g_1(k,\ell) \left(\sum_{d_2 \in D} g_2(k,d_2) |x_{d_2}|\left (\sum_{j=1}^m |g_2(\ell,j) x_j|\right )
\right. \right. \nonumber \\
& & \left. \left. +\, \left (\sum_{j=1}^m |g_2(k,j) x_j|\right )\left (\sum_{h=1}^m |g_2(\ell,h) x_h|\right ) \right ) \right ) \nonumber \\
& = & O\left ( \sum_{k =n}^\infty \sum_{\ell = n}^\infty \sum_{d \in D}  \sum_{j=1}^m g_2(d,k) g_1(k,\ell) g_2(\ell,j) |x_d| (\max_{j \leq n} |x_j|) \right . \nonumber \\
&    & \left. +\,  \sum_{k =n}^\infty \sum_{\ell = n}^\infty \sum_{j =1}^m  \sum_{h=1}^m g_2(j,k) g_1(k,\ell) g_2(\ell,h) (\max_{j \leq n} |x_j|)^2\right ) \label{ln:manysums}
\eea

Now, consider the quantity from line \eqref{ln:manysums} with $n=1$ and $N = \infty$. Applying the recursive definition of $g_n$, it satisfies
\bea 
& = & O\left ( \sum_{d \in D} \sum_{j = 1}^m g_4(d,j)|x_d| \max_{j \leq n} |x_j| + \sum_{j=1}^m \sum_{h=1}^m g_4(j,h) \max_{j \leq n} |x_j|^2 \right), \label{ln:termsumline}
\eea
where both double sums are convergent by our assumptions. Thus, as $n \to \infty$ in line \eqref{ln:manysums}, the sums decrease to zero (since as $n \to \infty$ all terms are removed from the sum), and so we have demonstrated that for any fixed $m$, the derivative in question is bounded by a quantity of the form
$$ o_n(1) (\max_{j \leq m} |x_j| + (\max_{j \leq m} |x_j|)^2)$$
(where $o_n$ means as $n \to \infty$). Combining this with line \eqref{ln:varbnd} gives
\bea 
& & \var(\mathbb{P}(E|X_1,...,X_m,X_D)|X_D=x_D) \nonumber \\
& \leq & \mathbb{E}\left [\left (\int_0^1 \frac{d}{dt}|_{t=y} \mathbb{P}(E|(X_1,...,X_m)=t(x_1,...,x_m),X_D=x_D)dy\right )^2\Bigg| X_D=x_D\right ] \nonumber \\
& = & o(1){E}(\max_{j\leq m} |x_j|)+(\max_{j\leq m} |x_j|)^2|X_D=x_D], \nonumber \eea
demonstrating the existence of the desired function $g_{m,D,x_D}(n)$.
\end{proof}

We are now in a position to immediately prove Theorem \ref{thm:gaussianMR}. 

\begin{proof}[Proof of Theorem \ref{thm:gaussianMR}] Combining Theorem \ref{thm:d1d2thm} with Lemmas \ref{lem:d1} and \ref{lem:d2} allows us to deduce (P5*), concluding the proof of Theorem \ref{thm:gaussianMR}.
\end{proof}

\begin{rek} The proofs of Theorem \ref{thm:density} and Lemma \ref{lem:d1} only require assumption (2) on $g_n$ for $n \leq 2$, but the proof of Lemma \ref{lem:d2} however requires this assumption up to $n = 4$. 
\end{rek}

\subsection{Applications and Examples of Infinite-Dimensional Gaussian Graphical Models}

We now illustrate the broad applicability of the theory developed above by means of three examples. To do so, we verify (P5*) by checking the conditions of Theorem \ref{thm:gaussianMR}. These entail (1) a uniform eigenvalue bound on the covariance matrices of the marginals, and (2) a covariance decay condition. For a given class of models, the decay condition is relatively straightforward to establish as compared to the eigenvalue condition. As a result, we will emphasize the verification of the eigenvalue condition in the examples below.

\begin{exa} \label{exa:autoreg} \textbf{(Infinite-dimensional autoregressive model.)} \\ 
Let $X_n$ by a collection of random variables defined by 
$$X_n = \sum_{j=1}^N \beta_{nj} X_{n-j} + \epsilon_n,$$
where $X_0, X_{-1}, ..., X_{-N+1} = 0$, and the $(\epsilon_n)_{n \in \mathbb{N}}$ are i.i.d.\ $\mathcal{N}(0,1)$. Assume that there exists $\delta > 0$ such that for all $n$ we have $\sum_{j=1}^N |\beta_{nj}| < 1 - \delta$. 

The conditions of Theorem \ref{thm:gaussianMR} can be verified for these $X_n$, and so the $X_n$ satisfy properties (IIP), (DCP), and (P5*). It can also be shown that they satisfy (P*) with respect to the graph $\mathcal{G} = (V,E)$ on $\mathbb{N}$ for which $(i,j) \in E$ iff $|i-j| \leq N$. Theorem \ref{thm:CIREquiv} thus implies that the $X_i$ satisfy (G*) with respect to this graph as well. The technical details are contained in the Supplemental section.
\end{exa}

\begin{exa} \label{exa:dd} \textbf{(Diagonally dominant Gaussian processes.)} \\
Let $(X_n)_{n \in \mathbb{N}}$ be a Gaussian process and let $\sigma_{ij} = \cov(X_i,X_j)$. Suppose the $\sigma_{ij}$ are uniformly diagonally dominant, in the sense that there exists some $\epsilon > 0$ such that
\be |\sigma_{ii}| > \epsilon + \sum_{j \neq i}|\sigma_{ij}| \ \ \forall i. \label{ln:bd1} \ee
Suppose also that there is a constant $C$ such that
\be |\sigma_{ii}| \leq C\ \ \forall i, \label{ln:bd2} \ee
and that the function $g_0(i,j) := |\sigma_{ij}|$ and the corresponding $g_1,g_2,g_3,$ and $g_4$ satisfy the bounds of Theorem \ref{thm:gaussianMR}, i.e.\ the easily verifiable covariance decay condition alluded to above.

Combining the uniform diagonal dominance with the Gershgorin circle theorem, we obtain the lower bound of $\epsilon$ on all eigenvalues of $\Sigma_n$ for any $n$. In addition, conditions \eqref{ln:bd1} and \eqref{ln:bd2} combine to show that $2C$ is an upper bound on all row sums of $\Sigma_n$ for any $n$ (where $\Sigma_n$ is the covariance matrix of $(X_1,...,X_n)$), and so is an upper bound on all eigenvalues of $\Sigma_n$ for any $n$. Thus, the conditions of Theorem \ref{thm:density} are satisfied, and so the $(X_n)_{n \in \N}$ satisfy (P5*).

Conditions \eqref{ln:bd1}, \eqref{ln:bd2} and the condition on the $g_i$ are satisfied, for example, if each $n \in \mathbb{N}$ has integer coordinates $c(n) \in \mathbb{Z}^m$, and the covariance between $X_i$ and $X_j$ is determined by applying the powered exponential covariance function to $d(c(i),c(j))$, the Euclidean distance between $c(i)$ and $c(j)$, for certain values of the covariance function parameters. More explicitly, this is the case if the $\sigma_{ij}$ satisfy
$$\sigma_{ij} = \exp(-d(c(i),c(j))^\alpha/V),$$
where $0 < \alpha \leq 2$ and $V$ is positive and sufficiently small (depending on $\alpha$). For more information on this family of covariance functions, known as the powered exponential family, see \cite{Yaglom}.

For the technical details of this example, see the Supplemental section.
\end{exa}

\begin{exa} \label{exa:cov} \textbf{(Gaussian lattice model.)} \\
Let $(X_p)_{p \in \Z^r}$ be a Gaussian process indexed by $\Z^r$ (i.e., on a lattice) satisfying
\bea \sigma_{p_1 p_2} & := & \cov(X_{p_1},X_{p_2})  =  \exp(-d(p_1,p_2)^2/V), \nonumber
\eea
where the function $d$ represents Euclidean distance, and $V > 0$ is arbitrary (so that the covariances are specified by a Gaussian covariance function). Note that unlike the previous example, the above covariances are in general not diagonally dominant.

The Gaussian process with the above covariance function can be shown to satisfy (IIP), (DCP), and (P5*) by verifying the conditions of Theorem \ref{thm:gaussianMR}. For the technical details, see Appendix \ref{app:MR}.
\end{exa}

\begin{rek}
Example \ref{exa:cov} is intended to demonstrate that many processes of interest satisfy (IIP), (DCP), and (P5*). For such processes, if (P*) is satisfied, then (G*) will be satisfied, allowing the notion of an infinite-dimensional graphical model. However, we wish to clarify that our intention in this example is not necessarily to obtain a process for which (P*) is satisfied. 
\end{rek}

\section{Graphical Discrete Processes}\label{sec:discproc}

\subsection{Graphical Models for Infinitely Many Discrete Random Variables}

\begin{defi}
A collection of random variables $\{X_i \ | \ i \in I\}$ is called a \emph{discrete process} provided that each $X_i$ takes values in $\{0,1\}$.
\end{defi}

\begin{rek}
While we restrict our attention in this paper to discrete processes taking on only two values, standard arguments allow the following results to be extended to the setting where the random variables may take one of $r$ distinct values for any finite $r > 2$.
\end{rek}

We now prove Theorem \ref{thm:discreteMR}, which is the central result of this section.

\begin{proof}[Proof of Theorem \ref{thm:discreteMR}]
We first use property (1) to obtain a positive density, which is used to prove (IIP).

Let $A, B, C, D$ be disjoint subsets of $\mathbb{N}$, and let $A, B, C$ be finite. If $D$ is finite, assumption (1) implies that $(X_A,X_B,X_C,X_D)$ has an everywhere-positive density with respect to the counting measure on $\{0,1\}^{|A|+|B|+|C|+|D|}$, and so (P5) will hold. Thus, we may assume $D$ is infinite. Let $m = |A| + |B| + |C|$.

We have by assumption that 
$$c_I < \mathbb{P}((X_A,X_B,X_C) = (a_1,...,a_m) \ | \ X_{d_1}=b_1,...,X_{d_r}=b_r)$$
for all $r$ and any $\{0,1\}$ valued sequence $(b_n)_{n \in \mathbb{N}}$. By the L\'e{}vy zero-one law, 
 \bea & & \lim_{r\rightarrow \infty} \mathbb{P}((X_A,X_B,X_C) = (a_1,...,a_m) \ | \ X_{d_1},...,X_{d_r}) \nonumber \\ 
 & = &  \mathbb{P}((X_A,X_B,X_C) = (a_1,...,a_m) \ | \ X_D ) \nonumber
 \eea
with probability 1 with respect to the marginal distribution, and therefore $\mathbb{P}$-almost surely we have
$$\mathbb{P}((X_A,X_B,X_C) = (a_1,...,a_m) \ | \ X_D ) > c_I.$$

Note that since we only need the above line to hold with respect to the marginal distribution, it is enough that it hold for $\mathbb{P}$-a.e. value of $X_D$.

From here, the argument verifying (IIP) is similar to the proof of Lemma \ref{lem:d1}.

Lemma \ref{lem:d2new} shows that (DCP) holds under the second assumption, and Theorem \ref{thm:d1d2thm} verifies (P5*), finishing the argument.
\end{proof}

\subsection{Applications and Examples of Infinite-Dimensional Discrete Graphical Models}

\begin{exa}\label{exa:discmc} \textbf{(Two-state Markov Chain.)} \\
Let $(X_n)_{n \in \mathbb{N}}$ be a $\{0,1\}$-valued Markov chain, with transition probabilities
$$p_n = \mathbb{P}(X_{n+1} = 1 | X_n = 1) \textrm{\ \ \ \ and \ \ \ \ } t_n = \mathbb{P}(X_{n+1} = 1 | X_n = 0).$$
Suppose that $0 < \mathbb{P}(X_1 = 1) < 1$, and that $0 < p_n, t_n < 1$ for all $n \in \mathbb{N}$. Suppose also that
$$\sum_{n=1}^\infty (1-(p_n-t_n)) = \infty$$

Then the $\{X_n \ | \ n \in \mathbb{N}\}$ satisfy properties (IIP) and (DCP), and hence (P5*).
\end{exa}
\begin{proof}
The proof is an application of Theorem \ref{thm:discreteMR}, and is provided in the Supplemental section.
\end{proof}

In the previous example, the collection of random variables was assumed to have graphical structure (as it forms a Markov chain). We now provide an example which is more general in the sense that it makes no assumption of a graphical relationship between the variables in the discrete process, and which will be useful in Example \ref{exa:infising1} below.

\begin{exa} \label{exa:discreteexample} \textbf{(Sparse Countable Sequences.)} \\
Let $(X_n)_{n \in \mathbb{N}}$ be a $\{0,1\}$-valued discrete stochastic process, and suppose the following:
\ben
\item for all $m \in \N$, there exists $\epsilon_m > 0$ such that for any $\{0,1\}$-valued sequence $(i_1,i_2,...,i_m)$, and any finite $B \subseteq \mathbb{N}$ disjoint from $\{1,...,m\}$,
$$\mathbb{P}((X_1,...,X_m) = (i_1,...,i_m) | X_B = x_B) > \epsilon_m \textrm{, and}$$
\item for $\tau(\omega) := \#\{X_i(\omega) = 1\}$, we have $\tau < \infty$ almost surely.
\een
Then the $\{X_n \ | \ n \in \mathbb{N}\}$ satisfy properties (IIP) and (DCP).
\end{exa}
\begin{proof}
See Appendix \ref{app:MR}.
\end{proof}

\begin{rek}
By exchanging $X_i$ with $Y_i = 1-X_i$, for any $\{0,1\}$-valued sequence $Z_1,Z_2,...$, the above result can be generalized to the case where $\#\{X_i \neq Z_i\} < \infty$ almost surely.
\end{rek}

\subsection{The Ising Model}
\ \\
The Ising model is a well-known and well-studied family of models of finitely many discrete random variables with graphical structure \cite{Snell, Spitzer}. Even infinite-dimensional versions of the Ising model have been well-studied \cite{georgii1988gibbs}, but much of this study has focused around lattice-based models and physics applications.

Our goal in this subsection is to obtain a graph-based infinite-dimensional generalization of the Ising model distribution with the ultimate goal of verifying (IIP) and (DCP), and also (P*) with respect to some nontrivial graph. In order to obtain this generalization in a rigorous manner, we must first establish a number of results about limits of the finite Ising model distributions. Existence and uniqueness results for related distributions have been established by other authors (\eg as in \cite{Dobruschin}), but for the most part these results have relied on a lattice-based framework rather than one based on general graphs, as is our goal. Hence we provide a self-contained development of an infinite-dimensional generalization of the Ising model in our graph-based framework. Our primary purpose in this section is to demonstrate examples of infinite graphical models by appealing to Theorem \ref{thm:discreteMR}. The proofs of the necessary results establishing the existence of the infinite Ising model and other properties thereof are provided in the Supplemental section.

We begin by rigorously introducing the finite-dimensional Ising model.

\begin{defi}
The Ising model consists of a discrete process $X_1,...,X_n$ along with an additional variable $X_0 = 1$. The set of parameters is given by $\Theta = (\theta_{ij})_{(i,j) \in E}$ where $E$ is the edge set of a graph $\mathcal{G} = (V,E)$ (with $V = \{0,1,...,n\}$) such that every nonzero node has an edge to zero. The $X_i$ are then distributed according to the unnormalized density
\be U(X = (x_1,...,x_n)) := \exp\left(\sum_{(i,j) \in E} \theta_{ij} x_i x_j\right). \label{ln:probmeas}\ee
Since $U(X)$ is finite for any choice of $X$, this induces a (normalized) probability distribution $\mathbb{P} \propto U$ on $\{0,1\}^n$.
\end{defi}

It is straightforward to verify from line \eqref{ln:probmeas} above that for this distribution,
\be \mathbb{P}(X_j = 1 | X_{-j} = x_{-j}) = \frac{1}{1+\exp(-\theta_{j0} - \sum_{(j,k) \in E}\theta_{jk}x_k)}, \label{ln:17.30} \ee
where $X_{-j}$ represents the vector containing all of the $X_i$ except $X_j$. From line \eqref{ln:17.30} it is clear that the (finite) Ising model satisfies the Markov property (L), and therefore also (P), with respect to its graph $\mathcal{G}$. In the finite case, Corollary \ref{thm:PearlPaz} ensures that this model satisfies property (G) as well.

\begin{exa} \label{exa:infising1}
For appropriate $\Theta$ it is possible to directly normalize the measure obtained from line \eqref{ln:probmeas} even if the number of variables is infinite. In particular, if $\sum_{x \in \{0,1\}^\infty} U(x) < \infty$, then $\mathbb{P}$ can be obtained by normalization (for this to make sense, we define $\exp(-\infty) = 0$). We remark that if this sum is finite, then at most countably many terms can be nonzero, and so $\mathbb{P}$ will be a discrete measure. 

Suppose that the graph $\mathcal{G} = (\N \cup \{0\}, E)$ is such that every node except the zero node has finite degree. Let $\Theta$ be arbitrary such that $\theta_{jk} \leq 0$ for all $j$ and $k$, and so that $\theta_{k0} \leq -2\log k$. Then the $X_i$ obtained from the measure defined in line \eqref{ln:probmeas} satisfy (IIP) and (DCP), and therefore (P5*) by Theorem \ref{thm:d1d2thm}. In addition, it can be shown that these $X_i$ satisfy (P*) with respect to the induced subgraph of $\mathcal{G}$ excluding the zero node, and thus also satisfy (G*) with respect to this graph by Theorem \ref{thm:CIREquiv}. For the technical details of this claim, which involve an application of the result of Example \ref{exa:discreteexample}, see Appendix \ref{app:MR}.
\end{exa}

In the above example, we relied upon the convergence of the sum of the $U(x)$. We now give a more general treatment of our infinite-dimensional Ising model.

\begin{defi}\label{def:iim}
Let $\mathcal{G} = (V,E)$ be a graph with vertex set $V = \{0\} \cup \mathbb{N}$ such that every nonzero node has an edge to zero. Let $\Theta = (\theta_{ij})_{(i,j) \in E}$. Let $\mathcal{G}_n$ be the induced subgraph on $\{0,1,...,n\}$. 

Define the distribution $\mathbb{P}_n$ on $\{0,1\}^n$ to be the Ising model distribution of $X_1,...,X_n$ with graph $\mathcal{G}_n$ and the same parameter set $\Theta$. Given $1 \leq m \leq n$, define $\mathbb{P}^m_n$ to be the marginal distribution of $X_1,...,X_m$ under $\mathbb{P}_n$. If for all $v \in \{0,1\}^m$ the limit
\be \lim_{n\rightarrow \infty} \mathbb{P}^m_n((X_1,...,X_m)=v) \label{ln:limexist} \ee
exists, we define the distribution $\mathbb{P}^m$ on $\{0,1\}^m$ to be given by this limit. That is,
$$\mathbb{P}^m((X_1,...,X_m) = v) = \lim_{n\rightarrow \infty} \mathbb{P}^m_n((X_1,...,X_m)=v).$$

Finally, we define the infinite Ising model with graph $\mathcal{G}$ and parameter set $\Theta$ to be the distribution on $\{0,1\}^\infty$ with the infinite product $\sigma$-algebra which has the finite-dimensional distributions provided by the $\mathbb{P}^m$ and their marginalizations.
\end{defi}

For ease of notation, we will assume that for $(i,j) \not \in E$ we have $\theta_{ij} = 0$, so that we need not specify that sums be taken only over edges $(i,j) \in E$.

\begin{exa} \label{exa:infising2}
We now present another infinite extension of the Ising model. Assume that $\mathcal{G}$ is a graph with vertex set $\mathbb{N} \cup \{0\}$ such that every node has finite degree except $0$, and $\Theta = (\theta_{ij})_{(i,j) \in E}$ satisfies
$$\sum_{(i,j) \in E} |\theta_{i,j}| < \infty.$$
Under this assumption, there is a limiting joint distribution of the $X_i$ which generalizes the finite-dimensional Ising model, and moreover this distribution is such that
$$f(X) = \sum_{i=1}^\infty 2^{-i} X_i$$
has a density with respect to Lebesgue measure on $[0,1]$. In addition, (IIP), (DCP), and (P5*) are satisfied by these variable. These variables also satisfy (P*) with respect to $\mathcal{G}$, and so satisfy (G*) as well. Details are provided in Appendix \ref{app:MR}.
\end{exa}

\begin{rek}
It is also possible to demonstrate the existence of the limiting distribution from Definition \ref{def:iim} for other examples of $\Theta$ and $\mathcal{G}$. However, we have restricted our attention to the choices of $\Theta$ and $\mathcal{G}$ from Examples \ref{exa:infising1} and \ref{exa:infising2} for the sake of brevity, and also because our goal is to provide sufficient demonstration of the applicability of Theorem \ref{thm:discreteMR}.
\end{rek}

\begin{rek}
Many of the above ideas can also be extended to the case of the generalized log-linear model, where for a given graph $\mathcal{G}$,
$$\log \mathbb{P}(X) = \sum_{A \textrm{ a clique of }\mathcal{G}} f_A(X_A).$$
In particular, similar arguments demonstrate that graphical models satisfying (IIP) and (DCP) can be defined when either $\sum_{A \textrm{ a clique of }\mathcal{G}} f_A(X_A) = -\infty$ for all but countably many choices of $X$, or when the sum $\sum_{A \textrm{ a clique of }\mathcal{G}} |f_A(X_A)|$ is uniformly bounded over all values of $X$.
\end{rek}

\begin{rek}We note here that Georgii \cite{georgii1988gibbs} introduces and studies in great detail an infinite-dimensional formulation of the Ising model in which only the conditional distributions for the model are specified, rather than a joint distribution or any marginals. Georgii takes this approach with the end goal of explaining physical phenomena. For example, he demonstrates that the non-uniqueness of a joint distribution with certain specified conditional distributions can be associated with phase transitions. We are, however, more interested in working with a specific joint distribution and studying its conditional independences. This is why we have provided our own self-contained formulation of an infinite-dimensional generalization of the Ising model which actually determines a specific joint distribution with the desired conditional distributions.
\end{rek}

\noindent \textbf{Acknowledgements:} We thank Apoorva Khare for reading a draft of the manuscript and for giving useful suggestions. We also thank Amir Dembo for providing a useful reference on the Ising model and discussions about the background to the paper.
\bibliography{mybib}{}
\bibliographystyle{acm}

\clearpage
\appendix 
\section{Proofs of Main Results} \label{app:MR}

\subsection{Proofs from Section \ref{sec:GGP}}

\noindent \textbf{Proof of Theorem \ref{thm:density}.}
Suppose that $A \subseteq \mathbb{N}$ is any finite collection of nodes. The claim is trivial when $B$ is finite, so let $B = \{b_1,b_2,...\} \subseteq \mathbb{N}$ be an infinite collection of nodes, and without loss of generality, assume that $b_1 < b_2 < ...$. We will use the notation $B_n = \{b_1,...,b_n\}$.

By the formulas for the conditional distribution of the multivariate normal distribution, we have
\be X_A | X_{B_n} \sim \mathcal{N}\left (\mu_A + \Sigma_{AB_{n}}\Sigma_{B_n}^{-1}(X_{B_n} - \mu_{B_n}),\Sigma_{A|B_n}\right ). \label{ln:dist} \ee

Thus, we have for any event $E \in \sigma(X_A,X_{B_n})$ that
\be \mathbb{P}(E) = \int_E f_n(x_A,x_{b_1},...,x_{b_n}) dx_A d\nu(x_B),\label{ln:densityline} \ee
where $f_n(x_A,X_{B_n})$ is the density of the distribution from line \eqref{ln:dist} (where $X_{B_n}$ is given), and where $\nu$ is the probability measure determining the distribution of $X_B = (x_{b_1},x_{b_2},...)$.

Combining the eigenvalue bounds with the Cauchy interlacing theorem for Schur complements (as in \cite{Zhang}, for example), we obtain that $\Sigma_{A|B_n}$ is a sequence of matrices whose eigenvalues are all uniformly bounded above by $C$, and so the entries are as well. Thus, each entry is confined uniformly in $n$ to a closed and bounded (and therefore sequentially compact) interval. Thus there is a subsequence $n_j$ such that $\Sigma_{A|B_{n_j}}$ converges to a matrix $\Sigma_{A|B}$. (For this argument, it will not be necessary to verify that $\Sigma_{A|B}$ is independent of the choice of the $n_j$.) 

We now consider the conditional mean. Let $K_n$ be the diagonal matrix of the conditional standard deviations, \ie $K_n$ is diagonal with diagonal entries given by 
$$(K_n)_{ii} = \sqrt{\var(X_{b_i}|X_{b_1},...,X_{b_{i-1}},X_{b_{i+1}},...,X_{b_n})}.$$
Note that combining the Cauchy interlacing theorem (again for Schur complements) with the eigenvalue bounds yields $\sqrt{c} < (K_n)_{ii} < \sqrt{C}$. 
Now, consider the matrix 
$$\Sigma_{AB_n}\Sigma_{B_n}^{-1}K_n = \Sigma_{AB_{n}}K_n^{-1} K_n\Sigma_{B_n}^{-1}K_n.$$
As discussed in section 5.1.3 of \cite{Lauritzen}, the diagonal entries of $K_n\Sigma_{B_n}^{-1}K_n$ are 1, and the nondiagonal $(i,j)$ entry of this matrix is the negative of the partial correlation of $X_{b_i}$ and $X_{b_j}$ given the remaining $X_{b_\ell}$, and so is bounded by a constant multiple of the conditional covariance (where this multiple depends on the bounds on the conditional variances provided by $c$ and $C$), giving us more specifically that 
\be (K_n\Sigma_{B_n}^{-1}K_n)_{ij} = O\left (g_0(b_i,b_j) + \sum_{k =1}^\infty g_0(b_i,k)g_0(k,b_j)\right ) = O(g_1(b_i,b_j)). \nonumber \ee
Note that the above argument shows that all conditional covariances are bounded above by a constant multiple of $g_1$ (where this constant multiple depends on $c$). 

Next, we combine the fact that the entries of $K_n^{-1}$ are bounded above by $\sqrt{1/c}$ and that $(\Sigma_{AB_{n}})_{i,j} \leq g_0(a_i,b_j)$ to obtain $(\Sigma_{AB_{n}}K_n^{-1})_{i,j} =O(g_0(a_i,b_j))$.

Thus, we have the bound
\bea
|(\Sigma_{AB_{n}}\Sigma_{B_n}^{-1}K_n)_{i,j}| & = & \left|\sum_{k=1}^n (\Sigma_{AB_{n}}K_n^{-1})_{i,k} (K_n\Sigma_{B_n}^{-1}K_n)_{k,j}\right| \nonumber \\
& = & O\left(\sum_{k=1}^n g_0(a_i,k)g_1(k,b_j)\right) \label{ln:convsum} \\
& = & O(g_2(a_i, b_j)), \nonumber
\eea
where we have used the fact that $g_n(i,j)$ is increasing in $n$.

Because a countable product of sequentially compact spaces is sequentially compact, there is some subsequence $n_{j_k}$ of the $n_j$ such that each entry of $\Sigma_{AB_{n_{j_k}}}\Sigma_{B_{n_{j_k}}}^{-1}K_{n_{j_k}}$ converges as $k \to \infty$, and we have a bound on the limiting entries
$$(\Sigma_{AB}\Sigma_{BB}^{-1}K)_{i,j} = O(g_2(a_i,b_j)).$$
We wish to make it clear that we only attempt to define the expression $(\Sigma_{AB}\Sigma_{BB}^{-1}K)$ considered as a whole, with entries determined as limits in the manner described above. In particular, we do not define $\Sigma_{BB}^{-1}$ or any of the other terms individually.

The quantity $K_n^{-1}(X_{B_n} - \mu_{B_n})$ has a multivariate normal distribution, and each coordinate is an independent $\mathcal{N}(0,1)$ random variable with variance at most $C/c$. Since $K_n^{-1}$ is diagonal, we can define $K^{-1}(X_B - \mu_B)$ as the Gaussian process which, for any $n$, is equal to $K_n^{-1}(X_{B_n} - \mu_{B_n})$ in the first $n$ coordinates, and this is well defined by the Kolmogorov consistency theorem.

We have
\bea
\mu_{A|B_{n_{j_k}}} & = & \mu_{A} + \Sigma_{AB_{n_{j_k}}}\Sigma_{B_{n_{j_k}}}^{-1}\left (X_{B_{n_{j_k}}} - \mu_{B_{n_{j_k}}}\right ) \nonumber \\
& & = \mu_{A} + \Sigma_{AB_{n_{j_k}}}\Sigma_{B_{n_{j_k}}}^{-1}K_{n_{j_k}}K_{n_{j_k}}^{-1}\left (X_{B_{n_{j_k}}} - \mu_{B_{n_{j_k}}}\right ), \nonumber \\
& & \rightarrow \mu_{A} + \Sigma_{AB}\Sigma_{BB}^{-1}KK^{-1}(X_{B} - \mu_{B}) \eea
where the expression in the final line is a vector of convergent sums with probability one. To see that these sums are convergent, note that $(\Sigma_{AB}\Sigma_{BB}^{-1}K)_{k,\ell} = O(g_2(k,\ell))$ with probability one, the entries of $K^{-1}$ are bounded above and so by Proposition \ref{prop:normsumbnd} $(K^{-1}(X_{B} - \mu_{B}))_\ell = O(\ell^{\epsilon})$ with probability one, and finally, by assumption $\sum_{\ell = 1}^\infty  g_2(k,\ell) \ell^\epsilon < \infty$. 

Thus $\mu_{A|B_{n_{j_k}}}$ has a limit $\mu_{A|B}$ as $k \rightarrow \infty$ with probability one.

Now, if for any fixed choice of $x_B$ the function $f_{n_{j_k}}(x_A,x_B) = f_{n_{j_k}}(x_A,x_{B_{n_{j_k}}})$ is the multivariate normal density with mean $\mu_{A|B_{n_{j_k}}}$ and covariance matrix $\Sigma_{A|B_{n_{j_k}}}$ (so that $f(x_A,x_B)$ does not depend on the values of $x_{B \setminus B_{n_{j_{k}}}}$), we have for all $E \in \sigma(X_A,X_{B_{n_{j_k}}})$ that

\be \mathbb{P}(E) = \int_E f_{n_{j_k}}(x_A,x_B) dx_A\ d\nu(x_B). \nonumber \ee

Because of the convergence of the conditional mean and covariance matrix, we have pointwise that
$$f_{n_{j_k}}(x_A,x_B) \rightarrow f(x_A,x_B),$$ 
where $f(x_A,x_B)$ is the multivariate normal density with the corresponding limiting mean and covariance matrix. 

Next, by the uniform eigenvalue bounds on $\Sigma_{A|B_{n_{j_k}}}$, we have that $f_{n_{j_k}}$ is bounded above uniformly in $n_{j_k}$. Thus, we may apply the dominated convergence theorem and obtain that for all $E \in \sigma(X_A,X_{B_{n_{j_k}}})$, we have

\be \mathbb{P}(E) = \int_E f(x_A,x_B) dx_A\ d\nu(x_B). \label{ln:peint} \ee

Since this holds for arbitrary $n_{j_k}$, this equality holds for any $E \in \cup_k \sigma(X_A,X_{B_{n_{j_k}}})$. The set of events for which this equality holds forms a Dynkin system, and since $n_{j_k} \rightarrow \infty$, the equality holds for all $E \in \cup_n \sigma(X_A,X_{B_n})$. Since this is an algebra generating the $\sigma$-algebra $\sigma(X_A,X_B)$, we may conclude by Dynkin's $\pi$-$\lambda$ theorem that the equality in line \eqref{ln:peint} holds for all $E \in \sigma(X_A,X_B)$. 

Thus, we have shown $$X_A|X_B \sim \mathcal{N}(\mu_A + (f_B(X_B))_A, \Sigma_{A|B}).$$
\qed
\ \\

\noindent \textbf{Details of Example \ref{exa:cov}.}
We verify the conditions of Theorem \ref{thm:density} in the case where the dimension $r = 2$, and note that the proof for higher $r$ is similar. For the sake of consistency with the technical notation used elsewhere, we reindex the variables $\{X_z : z \in \Z^2\}$ by $\N$, so that we have a bijection $f: \N \to \Z^2$, and we will from now on denote $X_{f(n)}$ by $X_n$. Note that $f(n)$ still refers to the ``location'' of the variable $X_n$ for the sake of determining covariances.

The uniform upper bound on the eigenvalues of $\Sigma_m$ for each $m \in \N$ is easily verified by bounding the row sums, which is possible due to the Gaussian rate of decay in the covariances, and can be done uniformly over all rows due to the stationarity of the covariance function. The uniform Gaussian rate of decay also allows one to verify the required conditions on $g_r$ for $r \in \{0,1,2,3,4\}$. The arguments are similar to the derivation of line \eqref{ln:gbnd} in Example \ref{exa:autoreg} and we omit them here. However, we must still verify the uniform positive lower bound on the eigenvalues of the matrices $\Sigma_m$. 

From the Cauchy interlacing theorem for principal submatrices (as in \cite{Zhang}), it is enough to verify these lower bounds for a subsequence of the $\Sigma_m$, where we may reorder the $X_n$ to suit our purposes. Moreover, for the sake of verifying the lower bound on the eigenvalues of the $\Sigma_m$, we may assume without loss of generality (again by the Cauchy interlacing theorem) that the coordinate function $f: \N \to \Z^2$ is surjective. Let us order the $X_n$ so that the random variables in the set $\{X_1,...,X_{(2m+1)^2}\}$ correspond to the $(2m+1)^2$ nodes in a square centered at $(0,0)$ with side length $2m+1$, as in Figure \ref{fig:sigmam}.

\begin{figure}
\begin{center}  
\includegraphics[width=3in]{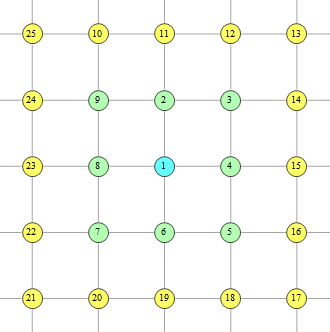}   
\caption{An arrangement of nodes appropriate for the argument in Example \ref{exa:cov}.} \label{fig:sigmam}
\end{center}  
\end{figure}

We now follow arguments similar to those in \cite{Gray} in order to obtain lower bounds on the eigenvalues of the matrices $\Sigma_{(2m+1)^2}$ in terms of a lower bound of a certain Fourier series. Gray works specifically with Toeplitz matrices, but the arguments generalize to this situation because the matrices $\Sigma_{(2m+1)^2}$ are block Toeplitz with Toeplitz blocks (higher-dimensional analogues also exist).

Let $f$ be the function on $[0,2\pi]^2$ defined by
$$f(x,y) = \sum_{j,k=-\infty}^\infty \exp\left (-\frac{j^2+k^2}{V}\right )e^{i(jx + ky)}.$$

This definition implies the relationship
$$ \exp\left (-\frac{j^2 + k^2}{V}\right ) = \frac{1}{4\pi^2} \int_0^{2\pi} \int_0^{2\pi} f(x,y) e^{-i(jx + ky)}dxdy.$$

We have for a length $(2m+1)^2$ vector $x$ that

\begin{eqnarray}
x^* \Sigma_m x & = & \sum_{j,k,r,s=-m}^{m} \exp\left(-\frac{(j-r)^2 + (k-s)^2}{V}\right) x_{(j,k)}x_{(r,s)}. \nonumber \\
& = & \sum_{j,k,r,s=-m}^{m}  \frac{1}{4\pi^2} \int_0^{2\pi} \int_0^{2\pi} f(x,y) e^{-i((j-r)x +( k-s)y)}dxdy\ x_{(j,k)}x_{(r,s)}. \nonumber \\
& = & \frac{1}{4\pi^2} \int_0^{2\pi}\int_0^{2\pi} f(x,y)\left | \sum_{j,r=0}^{m-1} x_{(j,r)} e^{ijx + iry} \right |^2 d x d y. \nonumber
\end{eqnarray}

Similarly,
$$
x^*x = \frac{1}{4\pi^2} \int_0^{2\pi}\int_0^{2\pi} \left | \sum_{j,r=0}^{m-1} x_{(j,r)} e^{ijx + iry} \right |^2 d xd y.
$$

Let $m_f = \essinf f(x,y)$ over $[0,2\pi]^2$, and similarly $M_f = \esssup f(x,y)$. Then combining the above, we have that if $\lambda$ is any eigenvalue of $\Sigma_m$,
$$m_f \leq \min_x \frac{x^* \Sigma_m x}{x^*x} \leq \lambda \leq \max_x \frac{x^* \Sigma_m x}{x^*x} \leq M_f.$$

Recall
\begin{eqnarray}
f(x,y) & = & \sum_{j,k=-\infty}^\infty \exp\left (-\frac{j^2+k^2}{V}\right )e^{i(jx + ky)} \nonumber \\
& = & \left ( \sum_{j=-\infty}^\infty \exp \left (-\frac{j^2}{V}\right )e^{ijx} \right) \left (\sum_{k=-\infty}^\infty \exp \left (- \frac{k^2}{V}\right )e^{iky} \right ), \nonumber 
\end{eqnarray}
and so if we define
$$ g(x,V) = \sum_{j=-\infty}^\infty \exp(- j^2/V)e^{ijx}, $$
then provided that $\min_x g(x,V) > 0$ we have that
$$m_f = \left (\min_x g(x,V)\right )^2$$
since $g(x,V)$ is a continuous function of $x$ for any given $V > 0$. 

Thus, in order to solve our problem, it is enough to show that for any given $V > 0$, the value $\min_x g(x,V)$ is positive. This function $g$ is a special case of the well-understood Jacobi theta function, which is defined by
$$\Theta(z|\tau) = \sum_{n=-\infty}^\infty e^{\pi i n^2 \tau}e^{2\pi i n z}.$$
\noindent In particular, we have that
$$g(x,V) = \Theta\left (\frac{x}{2\pi} \right | \left. \frac{i}{\pi V} \right ).$$
\noindent Thus a specific choice of (real) $V$ corresponds to $\tau = \frac{i}{\pi V}$, and any choice of real $x$ corresponds to $z = \frac{x}{2\pi}$, which is also real.

The zeros of the Jacobi theta function are known: $\Theta(z|\tau) = 0$ if and only if $ z = 1/2 + \tau/2 + n + m\tau$ for some $n,m \in \mathbb{Z}$ \cite{ss}. In particular, for $\tau$ purely imaginary with positive imaginary part, $\Theta(z | \tau)$ is nonzero for all real $z$. Thus, we may conclude that
$$g(x,V) \neq 0 \textrm{ for all } x \in [0,2\pi], V > 0.$$

But for any fixed $x$, the function $g(x,V)$ is real valued and continuous as a function of real $V > 0$, and for all $x$ we have $\lim_{V \to 0} g(x,V) = 1$. Since for any fixed real $x$ we know $g(x,V) \neq 0$ for any $V > 0$, by the intermediate value theorem we may conclude
$$g(x,V) > 0 \textrm{ for all } x \in [0,2\pi], V > 0.$$

Since $g(x,V)$ is a continuous function in $x$ for any fixed $V$, the image of the compact interval $[0,2\pi]$ under $g(\cdot,V)$ is also a compact interval, call it $[a,b]$. By the above argument, the lower endpoint of this interval must be positive, so in particular $g(x,V) \geq a > 0$ for all $x \in [0,2\pi]$. 

Thus we have shown that $\min_x g(x,V) \geq a > 0$ is positive, concluding the proof that the eigenvalues of $\Sigma_{(2m+1)^2}$ are uniformly positive, and thus allowing us to conclude from Theorem \ref{thm:density} that (IIP) and (DCP), and hence (P5*), are satisfied for this process.
\qed \ \\

\begin{rek}
The details of the above example were carried out over $\Z^m$ for $m=2$, but the same argument works for arbitrary dimension $m$ by considering an $m$-dimensional Fourier series, and again reducing the problem to a uniform positive lower bound of $g(x,V)$.
\end{rek}

\begin{rek}
The Jacobi theta function argument demonstrating the positivity of $m_f$ above leverages the form of the Gaussian covariance function. For other covariance functions, it is often straightforward to verify numerically a lower bound on $m_f$, and therefore on the eigenvalues of the $\Sigma_m$. 
\end{rek}\ \\ 

\subsection{Proofs from Section \ref{sec:discproc}}
\noindent \textbf{Details of Example \ref{exa:discreteexample}.}
We appeal to Theorem \ref{thm:discreteMR}. Point (1) of the statement of the example implies point (1) of Theorem \ref{thm:discreteMR}, so we just need to verify point (2).

Suppose that $A \in \sigma(X_n,X_{n+1},...)$, and that $m$ is fixed. Recall that 
$$\mathbb{P}((X_1,...,X_m) = (i_1,...,i_m)|X_B) > \epsilon_m$$
for any finite $B$, and so by the L\'evy zero-one law, we have $\mathbb{P}((X_1,...,X_m) = (i_1,...,i_m)|X_B) \geq \epsilon_m$ a.s.\ $X_B$ for any $I = (i_1,...,i_m)$ and any (potentially infinite) $B \subseteq \{m+1,m+2,...\}$. 

Let $B \subseteq \{m+1,m+2,...\}$ be arbitrary. Note that since $ \mathbb{E}[\tau] < \infty,$ we have that $\#\{b \in B | X_b(\omega) = 1\} < \infty$ almost surely, so we may restrict our attention to $X_B$ which are nonzero in only finitely many places.

Define the events
$$C_n := \{X_i = 0 \textrm{ for all }i \geq n, i \not \in B\}\textrm{, and}$$ 
$$D_n := C_n^c = \{X_i = 1 \textrm{ for some } i \geq n, i \not \in B\}.$$ 
Note that by the definition of $\sigma(X_n,X_{n+1},...)$, either $C_n \cap A = C_n$ or $C_n \cap A = \emptyset$. We treat these two cases separately. 

First, suppose $C_n \cap A = C_n$. Then
$$\mathbb{P}(C_n|X_1,...,X_m) \leq \mathbb{P}(A|X_1,...,X_m) \leq 1,$$
and so
\bea 
& &  \var(\mathbb{P}(A|(X_1,...,X_m)) = (i_1,...,i_m),X_B = x_B) \nonumber \\
& \leq & 1 - \mathbb{P}(C_n|(X_1,...,X_m) = (i_1,...,i_m),X_B=x_B) \nonumber \\
& =    & \mathbb{P}(D_n|(X_1,...,X_m) = (i_1,...,i_m),X_B = x_B) \nonumber \\
& =    & \frac{\mathbb{P}(D_n,(X_1,...,X_m) = (i_1,...,i_m)|X_B = x_B)}{\mathbb{P}((X_1,...,X_m) = (i_1,...,i_m) | X_B = x_B)} \nonumber \\
& \leq & \frac{\mathbb{P}(D_n|X_B = x_B)}{\mathbb{P}((X_1,...,X_m) = (i_1,...,i_m) | X_B = x_B)} \nonumber \\
& \leq & \frac{1}{\epsilon_m}\mathbb{P}(D_n|X_B = x_B) =: g_{m,B,x_B}(n). \nonumber
\eea

On the other hand, suppose that $C_n \cap A = \emptyset$. Then $A \subseteq D_n$, so 
$$0 \leq \mathbb{P}(A|(X_1,...,X_m) = (i_1,...,i_m)) \leq \mathbb{P}(D_n |(X_1,...,X_m) = (i_1,...,i_m)),$$ and so
\bea 
& &  \var(\mathbb{P}(A|(X_1,...,X_m)) = (i_1,...,i_m),X_B = x_B) \nonumber \\
& \leq & {P}(D_n|(X_1,...,X_m) = (i_1,...,i_m),X_B=x_B) \nonumber \\
& \leq & \frac{1}{\epsilon_m}\mathbb{P}(D_n|X_B = x_B) = g_{m,B,x_B}(n). \nonumber
\eea

Thus, to verify (2) of Theorem \ref{thm:discreteMR}, it is enough to check that ${P}(D_n|X_B = x_B)$ tends to zero as $n$ tends to infinity. As noted above, we may assume that $\{b \in B \ | \ x_b = 1\}$ is finite, and let $r = \max \{b \in B \ |\ x_b = 1\}$. Let $B_r = B \cap \{1,...,r\}$. Then for all $n > r$,
\be \mathbb{P}(D_n|X_B = x_B) \leq \mathbb{P}(D_n|X_{B_r} = x_{B_r}) \leq  \frac{1}{\epsilon_r} \mathbb{P}(D_n). \nonumber \ee
So it is enough to show that $\lim_{n \to \infty} \mathbb{P}(D_n) = 0$. 

By assumption, $\tau := \#\{X_i(\omega) = 1\} < \infty$ almost surely, and so $\max \{ i \ | \ X_i(\omega) = 1\} < \infty$ almost surely. Therefore,
\bea \mathbb{P}(D_n) & \leq & \mathbb{P}(\max \{ i \ | \ X_i(\omega) = 1\} \geq n) \nonumber \\
& \rightarrow & 0 \textrm{ as } n \rightarrow \infty,\nonumber \eea
concluding the proof. \qed
\ \\

\noindent \textbf{Details of Example \ref{exa:infising1}.}
For $\Theta$ satisfying these conditions, $\sum_{n=1}^\infty \theta_{k_n0} = -\infty$ for any choice of $k_n$, and so the only $x$ for which $U(x)$ could possibly be nonzero are those with only finitely many $x_i = 1$. 

We have that for some constant $C$, 
$$\sum_{k=n}^\infty \exp(\theta_{k0}) =: C_n \leq \sum_{k=n}^\infty \frac{1}{k^2} \leq \frac{C}{n}.$$

Because of this, we have that
\be \sum_{x} U(x) \leq  \sum_{n=1}^\infty \left[\sum_{k_1<...<k_n} \exp\left(\sum_{j=1}^n \theta_{k_j0}\right)\right] \leq  \sum_{n=1}^\infty \prod_{k=1}^n C_k \leq \sum_{n=1}^\infty \frac{C^n}{n!} \leq \exp(C). \nonumber
\ee

Thus we may normalize $U$ to a probability distribution $\mathbb{P}$ on $\{0,1\}^\infty$, and we note that under $\mathbb{P}$, the quantity $\#\{X_i = 1\}$ is almost surely finite. 

Next, for any finite $B$ disjoint from $\{1,...,m\}$, we have
\bea & & \mathbb{P}((X_1,...,X_m) = (i_1,...,i_m) | X_B = x_B) \nonumber \\
& \geq & \min_{x} \mathbb{P}((X_1,...,X_m) = (i_1,...,i_m) | X_{m+1}=x_{m+1},X_{m+2}=x_{m+2},...) \nonumber \eea

In the following lines, we will use the notation that for $n > m$ we have $i_n = x_n$ and $Y_n = x_n$. Then for a specific choice of $x$, we have
\bea & & \mathbb{P}((X_1,...,X_m) = (i_1,...,i_m) | X_{m+1}=x_{m+1},X_{m+2}=x_{m+2},...) \nonumber \\
& = & \frac{\exp(\sum_{(j,k) \in E, j\textrm{ or }k \in [1,m]} \theta_{jk}i_ji_k)}{\sum_{Y \in \{0,1\}^m} \exp(\sum_{(j,k) \in E, j\textrm{ or }k \in [1,m]} \theta_{jk} Y_j Y_k)}, \nonumber \eea
and since the graph under consideration has finite degree at each nonzero node, and $\theta_{j,k} > -\infty$ for each $(j,k) \in E$, the denominator in the above expression is bounded above uniformly in the choice of $x$, and so we have
$$\mathbb{P}((X_1,...,X_m) = (i_1,...,i_m) | X_B = x_B) > \epsilon_m$$
uniformly in $B$ for some positive $\epsilon_m$.

Thus, by the result of Example \ref{exa:discreteexample}, this collection of $X_i$ satisfies (IIP) and (DCP).

Note moreover that it is straightforward to directly verify that line \eqref{ln:17.30} holds for any (even infinite) collection of $(X_n)_{n \in \mathbb{N}}$ satisfying the above hypotheses, and so the $X_n$ satisfy (P*) with respect to the induced subgraph of $\mathcal{G}$ obtained by removing the zero node.
\qed
 \ \\

The following result is used in verifying the details of Example \ref{exa:infising2}. 

\begin{prop} \label{prop:ising}
Given a graph $\mathcal{G} = (\{0\}\cup\mathbb{N},E)$ and parameter set $\Theta$, the infinite Ising model (as in Definition \ref{def:iim}) with this graph and parameter set is well-defined and satisfies $\mathbb{P}((X_1,...,X_m) = v) > 0$ for all $m, v$ if and only if
$$\lim_{n \to \infty} f_m(v,n)$$
exists, and is finite and nonzero for all $v \in \{0,1\}^m$ and $m \in \mathbb{N}$, where
$$ f_m(v,n) := \frac{\sum_{X \in \{0,1\}^{n-m}}\exp\left(\sum_{i,j > m} \theta_{ij} X_i X_j + \sum_{i\leq m, j > m} \theta_{ij} v_i X_j\right)}
{\sum_{X \in \{0,1\}^{n-m}}\exp\left(\sum_{i,j > m} \theta_{ij} X_i X_j + \sum_{j > m} \theta_{0j} X_j\right)}.$$
\end{prop}

\begin{proof}
The proof is included in the Supplemental section.
\end{proof}\ \\

\noindent \textbf{Details of Example \ref{exa:infising2}.}
We first appeal to Proposition \ref{prop:ising} to obtain the existence of the joint limiting distribution.

By definition,
$$
f_m(v,n) =  \frac{\sum_{X \in \{0,1\}^{n-m}}\exp\left(\sum_{i,j > m}^n \theta_{ij} X_i X_j + \sum_{i\leq m, j > m}^n \theta_{ij} v_i X_j\right)}
{\sum_{X \in \{0,1\}^{n-m}}\exp\left(\sum_{i,j > m}^n \theta_{ij} X_i X_j + \sum_{j > m}^n \theta_{0j} X_j\right)},
$$
and so for sufficiently large $n$ (such that $n+1$ is not adjacent to any nonzero node with index less than or equal to $m$) 
$$
f_m(v,n+1) =  \frac{\displaystyle \sum_{X \in \{0,1\}^{n-m}}\exp\left(\sum_{i,j > m}^n \theta_{ij} X_i X_j + \sum_{i\leq m, j > m}^n \theta_{ij} v_i X_j\right)\left (1 + \exp \left ( \sum_{j \leq n} \theta_{(n+1)j} X_j \right )\right )}
{\displaystyle \sum_{X \in \{0,1\}^{n-m}}\exp\left(\sum_{i,j > m}^n \theta_{ij} X_i X_j + \sum_{j > m}^n \theta_{0j} X_j\right)\left (1 + \exp \left ( \sum_{j \leq n} \theta_{(n+1)j} X_j \right )\right )}.
$$

Define 
$$\alpha_n := \frac{f_m(v,n+1)}{f_m(v,n)}, \quad \beta_n := \sum_{j \leq n} |\theta_{nj}|.$$

Note that $\sum_n \beta_n < \infty$ by hypothesis. 

Thus, based on the above computations we have
\bea
|\alpha_n - 1| & \leq &  \max_{X,Y \in \{0,1\}^{n-m}} \left | 1 - \frac{1+\exp\left (\sum_{j\leq n} \theta_{(n+1)j}X_j\right )}{1+\exp\left (\sum_{j\leq n} \theta_{(n+1)j}Y_j\right )} \right | \nonumber \\
& \leq & \max\left  (\left | 1 - \frac{1+\exp(\beta_n)}{1+\exp(-\beta_n)}\right |, \left | 1 - \frac{1+\exp(-\beta_n)}{1+\exp(\beta_n)}\right |\right ) \nonumber \\ 
& = & \frac{\exp(\beta_n) - \exp(-\beta_n)}{1+\exp(-\beta_n)}. \nonumber
\eea

Noting that $\beta_n \to 0$, we have from the previous line that for all sufficiently large $n$, the quantity $|\alpha_n - 1| \leq 2 \beta_n$. Since $\sum \beta_n < \infty$, the above bound implies that the product $\prod_{n=1}^\infty \alpha_n$ exists and is nonzero and finite. Since $f_m(v,1) > 0$ and 
$$f_m(v,N+1) = f_m(v,1)\prod_{n=1}^N \alpha_n,$$
we may therefore conclude that the limit $\lim_{n \to \infty} f_m(v,n)$ exists and is nonzero for all $v$ and $m$. Thus, by Proposition \ref{prop:ising}, we obtain the existence of a limiting distribution which satisfies
\bea & & \mathbb{P}((X_{a_1},...,X_{a_n}) = (i_1,...,i_n) | (X_{b_1},...,X_{b_m}) = (j_1,...,j_m)) \nonumber \\
& \geq & \min_{j \in \{0,1\}^{\#\neigh(A)}} \mathbb{P}((X_{a_1},...,X_{a_n}) = (i_1,...,i_n) | X_{\neigh(A)} = j) \nonumber \\
& =: & \epsilon_n > 0 \nonumber \eea
uniformly over all finite choices of $B$ and $j$.

In addition, if $B$ is arbitrary and $A \in \sigma(X_B,X_n,X_{n+1},...,X_N)$ then 
\bea & & \var(\mathbb{P}(A|X_1,...,X_m,X_B)|X_B=x_B) \nonumber \\
& \leq & \max \mathbb{P}(A|X_1,...,X_m,X_B) - \min \mathbb{P}(A|X_1,...,X_m,X_B) \nonumber \\
& \leq & \max \mathbb{P}(A|X_{\neigh(A)}) - \min \mathbb{P}(A|X_{\neigh(A)}) \nonumber \\
& \leq & 1 - \frac{\exp(-\sum_{j,k > \min(\{n\}\cup\neigh\{n,...,N\})} |\theta_{jk}|)}{\exp(\sum_{j,k > \min(\{n\}\cup\neigh\{n,...,N\})} |\theta_{jk}|)}, \nonumber \eea
and this bound goes to zero as $n$ goes to infinity since each node in $\mathcal{G}$ is adjacent to at most finitely many others (and so $\min \neigh(A) \rightarrow \infty$), and since $\sum_{(j,k) \in E} |\theta_{jk}| < \infty$. 

Thus, both requirements of Theorem \ref{thm:discreteMR} are satisfied, and so this example satisfies (IIP) and (DCP), and also (P5*) by Theorem \ref{thm:d1d2thm}.
 
We now show that (P*) is satisfied. The same argument used in Example \ref{exa:infising1} allows us to compute the conditional distributions for each $\mathbb{P}_n$ and show that the $X_n$ satisfy (L*) with respect to the graph $\mathcal{G}_n = \mathcal{G} \cap \{1,...,n\}$ and the probability distribution $\mathbb{P}_n$ for all sufficiently large $n$. That is, for any finite $B \subseteq \mathbb{N}$ disjoint from $A$ and the neighbor set of $A$, we have $A \ci B \ | \ \neigh(A)$ with respect to $\mathbb{P}_n$ for all sufficiently large $n$, and therefore for the limiting $\mathbb{P}$ as well. Noting that the conditional independence of two infinite-dimensional processes is equivalent to the conditional independence of all of their finite-dimensional distributions, we obtain that (L*) is satisfied for the infinite-dimensional limiting distribution $\mathbb{P}$, and therefore that (P*) is as well.

Finally, we demonstrate the existence of the density for the quantity
$$f(X) = \sum_{i=1}^\infty 2^{-i} X_i.$$

For any $m < n$, we have
\be \mathbb{P}^m_{n}((X_1,...,X_m) = (x_1,...,x_m)) = \frac{\sum_{X : X_k = x_k, k \leq m} \exp(\sum_{(i,j) \in E}^n \theta_{ij} X_iX_j)}{\sum_{X \in \{0,1\}^n} \exp(\sum_{(i,j) \in E}^n \theta_{ij} X_iX_j)}, \nonumber \ee
and we can provide bounds on this quantity of the form
\be \mathbb{P^m}_{n}((X_1,...,X_m) = (x_1,...,x_m))  \leq  \frac{2^{n-m} \exp(\sum |\theta_{ij}|)}{2^n \exp(-\sum |\theta_{ij}|)} = \frac{\exp(2\sum|\theta_{ij}|)}{2^m}, \nonumber \ee
and also
\be \mathbb{P}^m_{n}((X_1,...,X_m) = (x_1,...,x_m)) \geq \frac{2^{n-m} \exp(-\sum |\theta_{ij}|)}{2^n \exp(\sum |\theta_{ij}|)}
= \frac{\exp(-2\sum|\theta_{ij}|)}{2^m}. \nonumber \ee

Thus, we have shown that for any $n > m$, we have a constant $C = \exp(2\sum|\theta_{ij}|)$ such that
$$\frac{1}{C2^m} \leq \mathbb{P}^m_{n}((X_1,...,X_m) = (x_1,...,x_m)) \leq \frac{C}{2^m},$$
and so this holds for the finite-dimensional distribution of the limiting distribution, $\mathbb{P}^m$, as well. This implies that the limiting distribution $\mathbb{P}^m$ is such that
$$f_m(X) := \sum_{i=1}^m 2^{-i} X_i$$
has a density with respect to Lebesgue measure for each $m$, and moreover that the limiting function $f$ and distribution $\mathbb{P}$ (obtained by letting $m \to \infty$) do as well, by taking the pointwise limit and appealing to the dominated convergence theorem.
\qed
\ \\

\subsection{Auxiliary Lemmas}
We conclude the appendix with three auxiliary lemmas that have been useful at various points in the paper.

\begin{lem}\label{lem:CEE}
Suppose that $\mathcal{F} \subseteq \mathcal{G}$ are $\sigma$-algebras on a probability space $(\mathbb{P},\mathcal{H},\Omega)$ such that for any $A \in \mathcal{G}$ there is some $A_F \in \mathcal{F}$ such that $\mathbb{P}(A \Delta A_F) = 0$. Then for any bounded measurable function $f$ we have $\mathbb{E}[f | \mathcal{F}] = \mathbb{E}[f | \mathcal{G}]$ almost surely.
\end{lem}
\begin{proof}
Since $\mathcal{F} \subseteq \mathcal{G}$, we have that $\mathbb{E}[f | \mathcal{F}]$ is $\mathcal{G}$-measurable. Now, let $S \in \mathcal{G}$ be arbitrary, and $S_F \in \mathcal{F}$ such that $\mathbb{P}(S \Delta S_F) = 0$. Then 
\be \mathbb{E}[1_S \mathbb{E}[f | \mathcal{F}]] = \mathbb{E}[1_{S_F} \mathbb{E}[f | \mathcal{F}]] = \mathbb{E}[1_{S_F} f] = \mathbb{E}[1_{S} f] \label{ln:ced1} \ee
where in line \eqref{ln:ced1} we have used the definition of conditional expectation to obtain the second equality, and for the other two the fact that $\mathbb{P}(S \Delta S_F) = 0$, which ensures that these steps cause the expression to change by an integral over a set of measure zero, and therefore preserve equality.
\end{proof}
\ \\

\begin{prop} \label{prop:normsumbnd}
Let $Z_1,Z_2,...$ be a collection of standard normal random variables (not necessarily i.i.d.). Then for any $\delta > 0$, we have
$$\mathbb{P}(Z_n < n^\delta \textrm{ for all sufficiently large }n) = 1.$$
\end{prop}
\begin{proof}
We have
\bea
& & \mathbb{P}(Z_n < n^\delta \textrm{ for all sufficiently large }n) = \mathbb{P}(\cup_{m=1}^\infty \cap_{n=m}^\infty Z_n < n^\delta) =  \lim_{m\rightarrow \infty} \mathbb{P}(\cap_{n=m}^\infty Z_n < n^\delta) \nonumber \\
& = & \lim_{m\rightarrow \infty} 1 - \mathbb{P}(\cup_{n=m}^\infty Z_n \geq n^\delta) \geq  1 -  \lim_{m\rightarrow \infty} \sum_{n=m}^\infty \mathbb{P}(Z_n \geq n^\delta). \label{ln:sumlim}
\eea
Thus, to show the claim it is enough to show that 
\be \sum_{n=1}^\infty \mathbb{P}(Z_n \geq n^\delta) < \infty, \label{ln:mustconverge} \ee
since then the final limit from line \eqref{ln:sumlim} will be zero.

By the standard tail bounds for the normal distribution, we have
\be 
\sum_{n=1}^\infty \mathbb{P}(Z_n \geq n^\delta) \leq \sum_{n=1}^\infty \frac{1}{n^\delta\sqrt{2\pi}}\exp(-n^{2\delta}/2) \leq \sum_{n=1}^\infty \exp(-n^{2\delta}/2). \label{ln:termcomp} \ee
Now, a comparison with $1/n^2 = \exp(-2\log(n))$ shows that the terms from line \eqref{ln:termcomp} are eventually bounded above by the terms of this convergent series, and so we have verified line \eqref{ln:mustconverge}, and hence the claim.
\end{proof}

\begin{prop} \label{prop:algapprox}
Let $(\Omega,\mathcal{F},\mathbb{P})$ be a probability space, and let $\mathcal{A} \subseteq \mathcal{F}$ be an algebra generating $\mathcal{F}$. Then for all $B \in \mathcal{F}$ and $\epsilon > 0$, we can find $A \in \mathcal{A}$ such that
$$\mathbb{P}(A\Delta B) < \epsilon.$$
\end{prop}

\begin{proof}
This is part of a hint to exercise 1.12.102 in \cite{Bogachev}. 
\end{proof}

\clearpage
\section{Supplemental Section}

\begin{prop}\label{prop:P5}
Suppose $V$ is finite, and that $\{X_v \ | \ v \in V\}$ is a collection of random variables. Then the following are equivalent:
\bi
\item (P5) For $X,Y,Z,W$ any finite collections of the $X_v$, if $X \ci Y \ | \ (W,Z)$ and $X \ci W \ | \ (Y,Z)$, then $X \ci (Y,W) \ | \ Z$.
\item (P5*) For $X, Z, A_1,...,A_n$ any finite collections of the $X_v$, if $X \ci A_i \ | \ (Z,(A_j)_{j \neq i})$ for all $i \leq n$, then $X \ci (A_1,...,A_n) \ | \ Z$.
\ei
\end{prop}
\begin{proof}
First, note that by Lemma \ref{lem:p1p4}, (P1) -- (P4) hold for the given collection of random variables.
\ \\
\noindent (P5*) $\Rightarrow$ (P5):\\
This follows from the case $n = 2$.\\
\ \\
\noindent (P5) $\Rightarrow$ (P5*):\\
We shall proceed by induction on $n$. (P5) gives the base case $n = 2$. Suppose now that (P5*) holds for any collection of subsets $A_i \subseteq V, i \leq n$.

Let $A_i, i \leq n+1$ be arbitrary finite collections of the $X_v$. Then by (P5), 
$$X \ci (A_i,A_{n+1}) \ |  \ (Z,(A_j)_{j \neq i,n+1}) \textrm{ for all } 1 \leq j \leq n.$$
Let $B_i = (A_i,A_{n+1})$. Then $X \ci B_i \ | \ (Z,(B_i)_{i \leq n})$, and so by the induction hypothesis, $X \ci (B_1,...,B_n) \ | \ Z$. By property (P2), $X \ci (A_1,...,A_{n+1}) \ |\ Z$. This completes the induction step.
\end{proof}

\begin{rek}
The above proof demonstrates that the equivalence of (P5*) and (P5) holds more generally at the level of ternary relations, as opposed to just the relation induced by conditional independence. 
\end{rek}

\begin{exa} \label{exa:5.1}
Let $Y_0, Y_1, Y_2$ be i.i.d.\ $\{0,1\}$-valued Bernoulli random variables with $p = 1/4$, and let $Y_3 = Y_1 + Y_2$. Let $(X_n)_{n > 3}$ be a collection of independent Bernoulli random variables, which are also independent of the $Y_i$, such that $\sum_{n=4}^\infty X_n < \infty$ with probability one. Finally, let 
$$X_1 = Y_1 + \sum_{k=1}^\infty X_{3k+1} \mod 2,$$
$$X_2 = Y_2 + \sum_{k=1}^\infty X_{3k+2} \mod 2,$$
$$X_3 = Y_3 + \sum_{k=1}^\infty X_{3k+3} \mod 2.$$
Also, let $X_0 = X_1 + Y_0$.
Then the $(X_k)_{k=0}^\infty$ form a discrete stochastic process. Conditioning on the $\sigma$-algebra generated by $\{X_4,X_5,...\}$ gives 
$$X_1 = Y_1 + c_1(X_4,X_5,...) \mod 2,$$ 
$$X_2 = Y_2 + c_2(X_4,X_5,...) \mod 2,$$ 
$$X_3 = Y_3 + c_3(X_4,X_5,...) \mod 2,$$ 
for $c_1,c_2,c_3$ some ${0,1}$-valued functions of $(X_i)_{i=4}^\infty$ so that
$$X_3 = X_1 + X_2 + c_3(X_4,...) - c_2(X_4,...) - c_1(X_4,...) \mod 2.$$

Because of the above relation, we have that
\be X_0 \ci X_1 \ | \ X_2, X_3, ..., \label{ln:ci1} \ee
and
\be X_0 \ci X_2 \ | \ X_1, X_3, ..., \label{ln:ci2} \ee
since $X_1$ and $X_2$ are constant given the remaining variables. 
However, it is \emph{not} the case that
$$X_0 \ci (X_1, X_2) \ | \ X_3, X_4, ...$$
since the conditional distribution of $(X_1,X_2)$ is such that $X_1$ takes both values 0 and 1 with nonzero probability, and $Y_0$ is 1 with probability 1/4, so that $X_0$ is equal to $X_1$ with probability 3/4 and different with probability 1/4. The lack of this conditional independence despite lines \eqref{ln:ci1} and \eqref{ln:ci2} demonstrates that both (IIP) and (P5*) fail for this example. Moreover, in the above example (P*) $\Rightarrow$ (G*) does not hold, since if we let $\mathcal{G}$ be the graph defined by (P*), then the set of nodes $\{3,4,...\}$ separates $\{0\}$ and $\{1,2\}$, but if (G*) held this would imply $X_0 \ci (X_1, X_2) \ | \ X_3, X_4, ...$, which we just disproved.

Note also that (P5) holds for every finite subcollection of these variables by Proposition \ref{prop:P1P5} since every finite collection of $n$ of these variables has an everywhere-positive density with respect to the counting measure on $\{0,1\}^n$.

Finally, we demonstrate that any collection of $\{0,1\}$-valued random variables $X_n$ for which $\#\{X_n \neq 0\} < \infty$ with probability one (including the collection just described) necessarily satisfies (DCP). Suppose that $\{X_i \ | \ i \in \N\}$ is such a collection. Suppose also that $D \subseteq \mathbb{N}$ is arbitrary, and $E \in \cap_n \sigma(X_D, X_n, X_{n+1},...)$. With probability one, the infinite-dimensional vector of random variables $\vec{X} = (X_1,X_2,...)$ takes on one of only at most countably infinitely many values $\vec{x}$ (namely, those with only finitely many ones), and so the event $\widetilde{E} = \{\vec{x} \ | \ \vec{x} \in E \text{ and } \P(\vec{X} = \vec{x}) > 0\}$ satisfies $\P(E \Delta \widetilde{E}) = 0$. Let $n \not \in D$ be fixed, and use the notation that if $\vec{x} = (x_1,x_2,...,x_n,...)$, then $\vec{x}' = (x_1,x_2,...,x_{n-1},1-x_n,x_{n+1},...)$. Define the $\sigma$-algebra $\mathcal{F} := \{B \in \sigma(X_D,X_{n+1},...) \ | \ \vec{x} \in B \Rightarrow \vec{x}' \in B\}$. Then for any $m \in D$ or $m > n$, we have $\sigma(X_m) \subseteq \mathcal{F}$. Thus $\sigma(X_D, X_{n+1}, X_{n+2},...) \subseteq \mathcal{F}$. So, we have for any $n \not \in D$ that $\vec{x} \in \widetilde{E}$ implies $\vec{x}' \in \widetilde{E}$. Thus, if we define $\pi(\widetilde{E}) := \{\vec{x}_D \ | \ \vec{x} \in \widetilde{E}\}$ so that $\pi$ is the projection onto the coordinates corresponding to variables in $D$, then $\widetilde{E} = X_D^{-1}(\pi(\widetilde{E})) \in \sigma(X_D)$, and so we have directly verified that (DCP) holds.
\end{exa}

\begin{exa} \label{exa:5.2}
Let $\theta$ be a $\{0,1\}$-valued bernoulli random variable with probability 0.5. Let $(X_i)_{i \in \N}$ be a collection of i.i.d.\ $\mathcal{N}(0,1)$ random variables, and let $Y_i = X_i + \theta$.

Then the $Y_i$ satisfy (P*) with respect to the edgeless graph on $\N$, but do not satisfy (G*) with respect to this graph. This is evident from the fact that for any $i,j$, with probability one, $\lim_{n \to \infty} \frac{1}{n}\sum_{k=1, k \neq i,j}^n Y_k = \theta$, so that conditioning on any set of all but two of the $Y_k$ determines the value of $\theta$ with probability one, and $Y_i$ is independent of $Y_j$ given $\theta$. However, the statement that (G*) holds with respect to the edgeless graph on $\N$ is equivalent to the statement that all of the random variables in question are marginally independent. However, it is clear that $Y_i$ and $Y_j$ are not marginally independent since
\bea \cov(Y_i,Y_j) & = & \E[Y_iY_j] - \E[Y_i]E[Y_j] \nonumber \\
& = & 0.5\left (\E[X_i X_j] + \E[(X_i+1)(X_j+1)]\right ) - (0.5(\E[X_i] + \E[X_i+1]))^2 \nonumber \\
& = & 0.5 - 0.25 = 0.25 \neq 0. \nonumber \eea

Note that (DCP) is not satisfied in this example since the event $\theta = 1$ is contained in the $\sigma$-algebra generated by any infinite collection of the random variables, but is not contained (even up to a measure zero modification) in the $\sigma$-algebra generated by any finite collection of the variables. Thus, letting $D = \emptyset$ in the definition of (DCP) shows that (DCP) fails to hold. On the other hand, it is easy to see that (IIP) is satisfied: any finite collection of the random variables has an everywhere-positive joint density, so the required implications of (IIP) involving a finite conditioning set will all hold; if the conditioning set is infinite, then the arguments from the second paragraph of this example show that all of the non-conditioned variables are conditionally independent, so that any implication required by (IIP) which involves an infinite conditioning set will also hold.

More generally, for any graph $\mathcal{G}$ which contains a node $v$ with finite degree, and another node $w$ not adjacent to $v$, we may let $(X_i)_{i \in \N}$ be a gaussian process independent of $\theta$ which satisfies (G*) with respect to $\mathcal{G}$. If we again let $Y_i = X_i + \theta$, then the $Y_i$ satisfy (P*) with respect to $\mathcal{G}$, but do not satisfy (G*) with respect to $\mathcal{G}$, as seen by considering the separating set to be the neighbor set of $v$, and noting that $v$ and $w$ are not independent given any finite subset of variables (since then $\theta$ is not determined with probability 1).
\end{exa}

\begin{exa}\label{exa:5.3}
Let $(A_n)_{n=1}^\infty$ be a collection of infinitely many random variables such that for all $i,j \in \N$, it is \emph{not} the case that 
$$A_i \ci A_j \ | \ \{A_k\ | \ k \neq i,j\}.$$
Let $\theta$ be a $\{0,1\}$-valued Bernoulli random variable, independent of all other variables mentioned so far, with $p = 0.5$. Finally, let $X_n = A_n + \theta$. We will show that the collection of random variables $\{X_1,X_2,...\}$ satisfies (P5*) but not (DCP). 

Suppose that $\mathcal{G}$ is a graph for which (P*) is satisfied. Note that the events $\{\theta = 0\}$ and $\{\theta = 1\}$ are contained in the sigma algebra generated by any infinite collection of the $X_i$ (by the law of large numbers). Thus, $\cov(X_i, X_j \ | \ \{X_k \ | \  k \neq i,j\}) = \cov(A_i, A_j \ | \ \{A_k \ | \ k \neq i,j\}) \neq 0$, and so there must be an edge between $i$ and $j$ in $\mathcal{G}$. Then $\mathcal{G}$ is the complete graph on $\N$, and so (G*) is trivially satisfied. Therefore (P*) implies (G*), and so by Proposition \ref{prop:converse}, (P5*) holds for this collection of random variables. 

If (DCP) were to hold, then it would be the case that $\cap_n \sigma(X_n, X_{n+1},...) = \sigma(\emptyset)$, the trivial $\sigma$-algebra. However, $\cap_n \sigma(X_n, X_{n+1},...)$ contains the nontrivial events $\{\theta = 1\}$ and $\{\theta = 0\}$, and so (DCP) does not hold.

All that remains is to demonstrate that there exist collections of random variables $(A_n)_{n=1}^\infty$ such that there are no pairwise conditional independences, i.e., such that there is no relation of the form $A_i \ci A_j \ | \ \{A_k \ | \ k \neq i,j\}$.

Let $0 < \alpha < 1$, and define the $A_i$ as follows: let $(B_i)_{i=1}^\infty$ be a collection of i.i.d.\ $\mathcal{N}(0,1)$ random variables, let $A_1 = B_1$, and for $n > 1$, let $A_n = B_n + \alpha B_{n-1}$. Suppose that $i < j$, and note that 
\bea \cov(A_i,A_j \ | \ \{A_k \ | \ k \neq i,j\}) & = & \lim_{k \to \infty} \cov(A_i, A_j \ | \ A_1,...,\widehat{A_i},...,\widehat{A_j},...,A_k) \nonumber \\
& = & \cov(A_i, A_j \ | \ A_1,...,\widehat{A_i},...,\widehat{A_j},A_{j+1}) \nonumber \eea
where the final line was obtained by using $(A_1,...,A_j) \ci (A_{j+2},...) \ | \ A_{j+1}$. 
Next, recall that
$$\cov(A_i, A_j \ | \ A_1,...,\widehat{A_i},...,\widehat{A_j},A_{j+1}) = 0$$
if and only if $(\Sigma_{j+1})^{-1}_{ij} = 0$, where $\Sigma_{j+1}$ is the covariance matrix of $(A_1,...,A_{j+1})$. This covariance matrix is given by 
$$(\Sigma_{j+1})_{ik} = \begin{cases} 1, & i=k=1 \\ 1 + \alpha^2,  & i=k\neq 1 \\ \alpha, & |i-k| = 1 \\ 0, & |i-k| > 1. \end{cases}$$
From this formula, it is readily verified that
$$(\Sigma_{j+1})^{-1}_{ik} = \begin{cases} (-\alpha)^{i-k}\sum_{r=0}^{j+1-i} (-\alpha)^{2r}, & i \geq k \\ (-\alpha)^{k-i}\sum_{r=0}^{j+1-k} (-\alpha)^{2r}, & i < k. \end{cases}$$
Thus, for $i < j$,
$$(\Sigma_{j+1})^{-1}_{ij} = (-\alpha)^{j-i} + (-\alpha)^{j-i+2} \neq 0.$$
Thus, it is not the case that $A_i \ci A_j \ | \ \{A_{k}, k \neq i,j\}$, completing the proof.
\end{exa}

\noindent \textbf{Proof of Lemma \ref{lem:p1p4}.}
(P1*): This is trivial from the commutativity of multiplication and intersection in line \eqref{ln:indep}.

(P2*): This is trivial from the fact that $\sigma(Y) \subseteq \sigma(Y,W)$ and $\sigma(W) \subseteq \sigma(Y,W)$.

Before proving (P3*) and (P4*), we now show that if $\sigma(X) \ci \sigma(W) \ | \ \sigma(Z)$, then for any $A \in \sigma(X)$, we have $\mathbb{P}(A|\sigma(Z,W)) = \mathbb{P}(A|\sigma(Z))$. To see this, let $M = \{ R \cap S \ | \ R \in \sigma(Z), S \in \sigma(W)\}$, and note that $M$ is a $\pi$-system which generates $\sigma(Z,W)$. Also, the collection $\mathcal{D}$ of events $T$ satisfying
$$\mathbb{E}[1_T\mathbb{E}[1_A \ | \ Z]] = \mathbb{E}[1_T 1_A]$$
is readily verified to be a Dynkin system (\ie it is closed under complement and disjoint union). Thus, if we can show that $M \subseteq \mathcal{D}$, then by Dynkin's $\pi$-$\lambda$ theorem, we can conclude from the definition of conditional expectation that $\mathbb{P}(A|\sigma(Z,W)) = \mathbb{P}(A|\sigma(Z))$ (noting that $\mathbb{P}(A|\sigma(Z))$ is a $\sigma(Z,W)$-measurable function). 

If $R \in \sigma(Z)$ and $S \in \sigma(W)$ are arbitrary, then
\bea
& & \mathbb{E}[1_{R\cap S}\mathbb{E}[1_A \ | \ \sigma(Z)]] \nonumber \\
& & = \mathbb{E}[\mathbb{E}[1_{R\cap S}\mathbb{E}[1_A \ | \ \sigma(Z)] \ | \ \sigma(Z)]] \label{ln:tower1} \\
& & = \mathbb{E}[\mathbb{E}[1_A \ | \ \sigma(Z)] \mathbb{E}[1_R 1_S\ | \ \sigma(Z)]] \label{ln:known1} \\
& & = \mathbb{E}[\mathbb{E}[1_A \ | \ \sigma(Z)] \mathbb{E}[1_S \ | \ \sigma(Z)] 1_R ] \label{ln:known2}\\
& & = \mathbb{E}[\mathbb{E}[1_A 1_S \ | \ \sigma(Z)] 1_R ] \label{ln:indep1} \\
& & = \mathbb{E}[\mathbb{E}[1_A 1_S 1_R \ | \ \sigma(Z)] ] \label{ln:known3} \\
& & = \mathbb{E}[1_{R \cap S} 1_A], \label{ln:tower2}
\eea
where lines \eqref{ln:tower1} and \eqref{ln:tower2} are justified by the tower property, lines \eqref{ln:known1}, \eqref{ln:known2}, and \eqref{ln:known3} by removing what is known, and line \eqref{ln:indep1} by the independence assumption. Thus, we have $\mathbb{P}(A|\sigma(Z,W)) = \mathbb{P}(A|\sigma(Z))$.

(P3*): Suppose that $\sigma(X) \ci \sigma(Y,W) \ | \ \sigma(Z)$.
We now wish to show that for arbitrary $A \in \sigma(X)$ and $B \in \sigma(Y)$,
$$\mathbb{P}(A | \sigma(Z,W))\mathbb{P}(B | \sigma(Z,W)) = \mathbb{P}(A,B | \sigma(Z,W)).$$
By the definition of conditional expectation, it is enough to show that, for $C$ an arbitrary event in $\sigma(Z,W)$, we have
$$\mathbb{E}[ 1_C \mathbb{P}(A | \sigma(Z,W))\mathbb{P}(B | \sigma(Z,W))] = \mathbb{E}[1_C 1_A 1_B].$$
Well, 
\bea & & \mathbb{E}[ 1_C \mathbb{E}[1_A | \sigma(Z,W)]\mathbb{E}[1_B | \sigma(Z,W)]] \\
& & = \mathbb{E}[ \mathbb{E}[1_A | \sigma(Z)]\mathbb{E}[1_B 1_C | \sigma(Z,W)]] \label{ln:above} \\
& & = \mathbb{E}[ \mathbb{E}[\mathbb{E}[1_A | \sigma(Z)]\mathbb{E}[1_B 1_C | \sigma(Z,W)] \ | \ \sigma(Z)]] \label{ln:tower3}\\
& & = \mathbb{E}[ \mathbb{E}[1_A | \sigma(Z)] \mathbb{E}[\mathbb{E}[1_B 1_C | \sigma(Z,W)] \ | \ \sigma(Z)]] \label{ln:known5}\\
& & = \mathbb{E}[ \mathbb{E}[1_A | \sigma(Z)] \mathbb{E}[1_B 1_C | \sigma(Z)]] \label{ln:tower4}\\
& & = \mathbb{E}[ \mathbb{E}[1_A 1_B 1_C | \sigma(Z)]] \label{ln:indep2} \\
& & = \mathbb{E}[ 1_A 1_B 1_C], \label{ln:tower5}
\eea

\noindent where in line \eqref{ln:above} we use (P2*) and the above argument, and remove what is known. Line \eqref{ln:known5} is obtained by removing what is known, lines \eqref{ln:tower3}, \eqref{ln:tower4}, and \eqref{ln:tower5} are obtained by the tower property, and line \eqref{ln:indep2} is by the assumed independence. This verifies (P3*).

(P4*): Suppose that $X \ci Y \ | \ (Z,W)$ and $X \ci W \ | \ Z$. We wish to show $X \ci (Y,W) \ | \ Z$. As in the proof of (P3*), we use the fact that $X \ci W \ | \ Z$ to obtain that for any $A \in \sigma(X)$, we have $\mathbb{P}(A | \sigma(Z,W)) = \mathbb{P}(A | \sigma(Z))$.

Suppose now that $A \in \sigma(X), B \in \sigma(Y), C \in \sigma(W)$, and $D \in \sigma(Z)$ are arbitrary.

Then
\bea
& & \mathbb{E}[\mathbb{E}[1_A \ | \ Z]\mathbb{E}[1_{B \cap C} \ |\ Z] 1_D] \nonumber \\
& & = \mathbb{E}[\mathbb{E}[1_A \ | \ Z]\mathbb{E}[\mathbb{E}[1_B 1_C \ |\ Z,W] \ | \ Z] 1_D] \label{tower:a1} \\
& & = \mathbb{E}[\mathbb{E}[\mathbb{E}[1_A \ | \ Z] \mathbb{E}[1_B 1_C \ |\ Z,W] \ | \ Z] 1_D] \label{known:a1} \\
& & = \mathbb{E}[\mathbb{E}[\mathbb{E}[1_A \ | \ Z,W] \mathbb{E}[1_B 1_C\ |\ Z,W] \ | \ Z] 1_D] \label{fact:a1} \\
& & = \mathbb{E}[\mathbb{E}[\mathbb{E}[1_A \ | \ Z,W] \mathbb{E}[1_B \ |\ Z,W] 1_C 1_D \ | \ Z]] \label{known:a2} \\
& & = \mathbb{E}[\mathbb{E}[\mathbb{E}[1_A 1_B \ | \ Z,W] 1_C 1_D \ | \ Z]] \label{indep:a1} \\
& & = \mathbb{E}[\mathbb{E}[\mathbb{E}[1_A 1_B 1_C 1_D \ | \ Z,W] \ | \ Z]] \label{known:a3} \\
& & = \mathbb{E}[1_A 1_B 1_C 1_D ]. \label{tower:a2}
\eea

Thus, we may conclude by a $\pi$-$\lambda$ argument (similar to that used to show $\mathbb{P}(A|\sigma(Z,W)) = \mathbb{P}(A|\sigma(Z))$ prior to (P3*)) that $\sigma(X) \ci \sigma(Y,W) \ | \ \sigma(Z)$. \qed \ \\

\noindent \textbf{Proof of Lemma \ref{lem:d1}.}
Let $A, B, C \subseteq \mathbb{N}$ be finite, and let $D \subseteq \mathbb{N}$ be an arbitrary subset. Suppose that $X_A \ci X_B | X_C, X_D$ and $X_A \ci X_C \ | \ X_B, X_D$. If $D$ is finite, the verification of (IIP) is a consequence of Proposition \ref{prop:P1P5}, so assume $D \subseteq \mathbb{N}$ is infinite. 

By Theorem \ref{thm:density}, we have that for any $E \in \sigma(X_A, X_B, X_C, X_D)$, 
$$\mathbb{P}(E) = \int_E f(x_A,x_B,x_C,x_D) dx_A\ dx_B\ dx_C\ d\mu(x_D)$$
for some probability measure $\mu$, where $f(x_A,x_B,x_C,x_D)$ is the multivariate normal density with the mean $\mu_{A,B,C|D}(x_D)$ and covariance matrix $\Sigma_{A,B,C|D}$. 

Similarly, we have the existence of $f_1(x_B,x_C,x_D)$, $f_2(x_A,x_C,x_D)$, and $f_3(x_A,x_B,x_D)$, the corresponding densities for events in $\sigma(X_B,X_C,X_D)$, $\sigma(X_A,X_C,X_D)$, and $\sigma(X_A,X_B,X_D)$ respectively. Note that these are densities with respect to Lebesgue measure in the corresponding $A$, $B$, and $C$ coordinates, and $\mu$ for $x_D$, and moreover since they are normal densities, they are positive with probability one.

Let $E_A \in \sigma(X_A)$ and $E_B \in \sigma(X_B)$. By the conditional independence $X_A \ci X_B | X_C, X_D$, we have
$$ \mathbb{P}(E_A | X_C, X_D)\mathbb{P}(E_B | X_C, X_D) = \mathbb{P}(E_A, E_B | X_C, X_D). $$

Let $E = \{(x_a,x_b,x_c,x_d) : f_1(x_b,x_c,x_d)f_2(x_a,x_c,x_d) \neq f(x_a,x_b,x_c,x_d)\}$. Suppose that $\mathbb{P}(E) > 0$. Then $E = E_+ \cup E_-$, where 
\bea E_+ & = & \{(x_a,x_b,x_c,x_d) : f_1(x_b,x_c,x_d)f_2(x_a,x_c,x_d) > f(x_a,x_b,x_c,x_d)\},\textrm{ and }\nonumber \\
E_- & = & \{(x_a,x_b,x_c,x_d) : f_1(x_b,x_c,x_d)f_2(x_a,x_c,x_d) < f(x_a,x_b,x_c,x_d)\}. \nonumber \eea
At least one of these two sets has positive probability, so without loss of generality we will assume $\mathbb{P}(E_+) > 0$.

Then 
\be \int_{E_+} f_1(x_b,x_c,x_d)f_2(x_a,x_c,x_d) - f(x_a,x_b,x_c,x_d) dx_adx_bdx_cd\mu(x_d) > 0. \label{ln:intbnd} \ee
By Proposition \ref{prop:algapprox}, there is a sequence of sets $F_n$ which are finite unions of finite intersections of sets in $\sigma(X_A), \sigma(X_B), \sigma(X_C)$, or $\sigma(X_D)$, such that $\mathbb{P}(E_+ \Delta F_n) < 1/n$. Thus, since all of the densities under consideration are bounded (this can be seen from the eigenvalue bounds in the statement of Theorem \ref{thm:density}), we have from the dominated convergence theorem that 
$$\int_{F_n} f_1(x_b,x_c,x_d)f_2(x_a,x_c,x_d) - f(x_a,x_b,x_c,x_d)dx_adx_bdx_cd\mu(x_d)$$
tends to the expression on the left side of line \eqref{ln:intbnd}, and so is positive for sufficiently large $n$.

Since $F_n$ is a finite union of finite intersections, there is some $G = G_A \cap G_B \cap G_C \cap G_D$ with $G_A \in \sigma(X_A)$, etc., which is one of the terms in the union comprising $F_n$, and for which
$$\int_{G} f_1(x_b,x_c,x_d)f_2(x_a,x_c,x_d) - f(x_a,x_b,x_c,x_d) dx_adx_bdx_cd\mu(x_d) > 0.$$
The above integral can be written
\bea
0 & < & \int_{G_D \cap G_C} \int_{G_B} \int_{G_A} f_1(x_b,x_c,x_d)f_2(x_a,x_c,x_d) - \nonumber \\
  &   & - f(x_a,x_b,x_c,x_d) dx_adx_bdx_cd\mu(x_d) \nonumber \\
  & = & \int_{G_D \cap G_C} \int_{G_B} f_1(x_b,x_c,x_d)dx_b \int_{G_A}f_2(x_a,x_c,x_d) dx_a dx_c d\mu(x_d) \nonumber \\
  &   & - \int_{G_D \cap G_C} \int_{G_B} \int_{G_A} f(x_a,x_b,x_c,x_d) dx_adx_bdx_cd\mu(x_d) \nonumber \\
  & = & \int_{G_D \cap G_C} \left[ \mathbb{P}(G_B | X_C=x_c,X_D=x_d)\mathbb{P}(G_A | X_C=x_c,X_D=x_d)\right. \nonumber \\
	&   & - \left. \mathbb{P}(G_A \cap G_B | X_C=x_c,X_D=x_d)\right]dx_cd\mu(x_d). \nonumber 
\eea
This contradicts the independence assumption $X_A \ci X_B \ | \ X_C, X_D$, and therefore we have that $\mathbb{P}(E) = 1$. That is, with probability one
$$f_1(x_b,x_c,x_d)f_2(x_a,x_c,x_d) = f(x_a,x_b,x_c,x_d).$$
A similar factorization holds for $f_1$ and $f_3$. 

Because of this, with probability one it is the case that
\bea 
f_2(x_a,x_c,x_d) & = & \frac{f(x_a,x_b,x_c,x_d)}{f_1(x_b,x_c,x_d)} \nonumber \\
& = & \frac{f_3(x_a,x_b,x_d)f_1(x_b,x_c,x_d)}{f_1(x_b,x_c,x_d)} \nonumber \\
& = &  f_3(x_a,x_b,x_d). \nonumber 
\eea 
Thus almost surely $f_2$ and $f_3$ only depend on $x_a$ and $x_d$, so we may write $f_2(x_a,x_c,x_d) = f_2(x_a,x_d)$, and similarly for $f_3$. Using this we obtain 
\bea 
f(x_a,x_b,x_c,x_d) & = & f_1(x_b,x_c,x_d) f_2(x_a,x_c,x_d) \nonumber \\
& = & f_1(x_b,x_c,x_d) f_2(x_a,x_d) \nonumber 
\eea
with probability one. Thus $X_A \ci (X_B,X_C) \ | \ X_D$, completing the verification of (IIP). \qed \ \\

\noindent \textbf{Details of Example \ref{exa:autoreg}.}
Under these assumptions on the $X_i$, we claim that
$$\var(X_n) \leq \sum_{k=0}^{n-1} (1-\delta)^{2k}.$$
The base case $n=1$ is trivial, since $\var(X_1) = 1$. For the induction step, we have
\bea \var(X_{n+1}) & = & 1 + \var(\sum_{j=1}^N \beta_{nj} X_{n-j}) \nonumber \\
& \leq & 1 + (\sum_{j=1}^N |\beta_{nj}|)^2 \max_{1 \leq j \leq N}\var(X_{n-j}) \nonumber \\
& \leq & 1 + (1-\delta)^2 \max_{1 \leq j \leq N}\var(X_{n-j}) \nonumber \\
& \leq & 1 + (1-\delta)^2\left(\sum_{k=0}^{n-1} (1-\delta)^{2k}\right) \label{ln:byindhyp} \\
& = & \sum_{k=0}^{n} (1-\delta)^{2k}, \nonumber \eea
where we have used the induction hypothesis to obtain line \eqref{ln:byindhyp}. From this, we may conclude that for all $n$,
$$\var(X_n) \leq \left(\sum_{k=0}^\infty (1-\delta)^{2k}\right) = \frac{1}{1 - (1-\delta)^2} \leq \frac{1}{\delta}.$$

For $n > m$ we have
\bea \cov(X_n,X_m) & = & \cov(\sum_{j=1}^N \beta_{nj} X_{n-j},X_m) \nonumber \\
& = & \sum_{j=1}^N \beta_{nj} \cov(X_{n-j},X_m) \nonumber \\
& \leq & (1-\delta) \max_{1 \leq j \leq N} \cov(X_{n-j},X_m). \nonumber \eea

By iteratively applying the above inequality, we may keep decreasing the subscript on the first $X$ variable in the previous line until the subscript reaches $m$, and we gain a factor of $(1-\delta)$ each time. Since decreasing $n$ in this iterative manner until it is at most $m$ requires at least $(n-m)/N$ iterations, we obtain
\bea \cov(X_n,X_m) & \leq & (1-\delta)^{\frac{n-m}{N}} \max_{1 \leq j \leq N} \cov(X_{m+j},X_m) \nonumber \\
& \leq & (1-\delta)^{\frac{n-m}{N}} \var(X_m) \nonumber \\
& \leq & (1-\delta)^{\frac{n-m}{N}}\frac{1}{\epsilon}. \eea

Thus, we may set $g_0(n,m) = (1-\epsilon)^{\frac{|n-m|}{N}}\frac{1}{\epsilon}$, and note that
\bea
g_1(n,m) & = & \frac{(1-\delta)^{\frac{|n-m|}{N}}}{\delta} + \sum_{k=1} ^ \infty \frac{(1-\delta)^{\frac{|n-k| + |k-m|}{N}}}{\delta} \nonumber \\
& = & O\left( (|n-m|+2)\frac{(1-\delta)^{\frac{|n-m|}{N}}}{\delta} + \sum_{k=1}^\infty \frac{(1-\delta)^{\frac{|n-m|+k}{N}}}{\delta}\right ) \nonumber \\
& = & O\left (|n-m|\frac{(1-\delta)^{\frac{|n-m|}{N}}}{\delta}\right ), \nonumber
\eea
and more generally the same argument gives
\be
g_r(n,m) =  O_r\left(|n-m|^r\frac{(1-\delta)^{\frac{|n-m|}{N}}}{\delta}\right ). \label{ln:gbnd}
\ee
For these $g_r$, the requirements of Theorem \ref{thm:density} are clearly satisfied, so all that remains is to demonstrate the required eigenvalue bounds.

The maximum eigenvalue of the matrix $\Sigma_n = (\sigma_{ij})_{1\leq i,j \leq n}$ is bounded by its maximum row sum, which is bounded by
\be \max_{1 \leq i \leq n} \sum_{j=1}^n \frac{(1-\delta)^{\frac{|i-j|}{N}}}{\delta} \leq 2 \sum_{k=0}^\infty \frac{(1-\delta)^{k/N}}{\delta} \leq \frac{2}{\delta (1-\delta)^{1/N}} < \infty. \nonumber \ee

Finally, we verify an upper bound on the eigenvalues of the inverses of the covariance matrices. Let $\sigma_{ij} = \cov(X_i,X_j)$, let $\Sigma_n = (\sigma_{ij})_{i,j \leq n}$, and define $\Sigma_n^{-1} = (\sigma^{ij})_{i,j \leq n}$.

As discussed in Section 5.1.3 of \cite{Lauritzen}, for any $i,j \leq n$ we have
$$\sigma^{ii} = \var(X_i|X_{-i})^{-1},$$
and
$$\frac{\sigma^{ij}}{\sqrt{\sigma^{ii}\sigma^{jj}}} = \frac{\cov(X_i,X_j|X_{-\{i,j\}})}{\sqrt{\var(X_i|X_{-\{i,j\}})\var(X_j|X_{-\{i,j\}})}},$$
which is a partial correlation and so is bounded in magnitude by 1. (Recall the notation $$X_{-i} = (X_1,...,X_{i-1},X_{i+1},...,X_n),$$ and similarly for $X_{-\{i,j\}}$.) Thus
\bea 
|\sigma^{ij}| & \leq & \sqrt{\sigma^{ii}\sigma^{jj}} \nonumber \\
& \leq & \sqrt{\var(X_i|X_{-i})^{-1}\var(X_j|X_{-j})^{-1}}. \nonumber 
\eea

We now claim that for any $i$ and $n$ with $i \leq n$, the quantity $\var(X_i|X_{-i})^{-1}$ is bounded above by a uniform constant $C(\delta,N)$. To see this, note that $\var(X_i | X_{-i})$ depends only on $\beta_{kj}$ for $i \leq k \leq i+N$ and $0 \leq j \leq N$, and is nonzero for any choice of the $\beta_{kj}$ with $\sum_{j=1}^N |\beta_{kj}| \leq 1-\delta$. Moreover, $\var(X_i | X_{-i})$ is a continuous function of $\beta$, and the condition $\sum_{j=1}^N |\beta_{kj}| \leq 1-\delta$ defines a compact domain for $\beta$. Thus the continuous function $\var(X_i | X_{-i})$ of $\beta$ has a compact image which does not include zero, and so for any choice of the $\beta_{ij}$, the quantity $\var(X_i | X_{-i})^{-1}$ is bounded above by some constant $C(\delta,N)$. Therefore we may conclude that $|\sigma^{ij}| \leq C(\delta,N)$.

Next, suppose $j-i > N$. We have $X_i \in \sigma(Y_1,...,Y_i)$, and the $Y_n$ are i.i.d., so $X_i \ci Y_j$. Also, $X_{-\{i,j\}}$ contains $X_{j-1}, X_{j-2}, ..., X_{j-N}$, so 
\bea 
&   & \mathbb{P}(X_j < a, X_i < b | X_{-\{i,j\}}) \nonumber \\
& = & \mathbb{P}\left (\sum_{k=1}^N \beta_{jk} X_{j-k} + \epsilon_j < a, X_i < b \ \Bigg|\  X_{-\{i,j\}}\right ) \nonumber \\
& = & \mathbb{P}\left (\epsilon_j < a -\sum_{k=1}^N \beta_{jk} X_{j-k}, X_i < b\ \Bigg|\ X_{-\{i,j\}}\right ) \nonumber \\
& = & \mathbb{P}\left (\epsilon_j < a - \sum_{k=1}^N \beta_{jk} X_{j-k} \ \Bigg|\ X_{-\{i,j\}}\right ) \mathbb{P}(X_i < b | X_{-\{i,j\}}) \label{ln:nonindep} \\
& = & \mathbb{P}(X_j < a | X_{-\{i,j\}}) \mathbb{P}(X_i < b | X_{-\{i,j\}}) \nonumber
\eea
where we obtain line \eqref{ln:nonindep} from the fact that $\epsilon_j$ is independent of $\epsilon_k$ for all $k < j$, and $X_k \in \sigma(\epsilon_1,...,\epsilon_{j-1})$ for all $k < j$. 

Thus $X_j$ and $X_i$ are conditionally independent given the other variables, and $\sigma^{ij} = 0$ when $|i-j| > N$. Therefore for any $n$, the matrix $\Sigma_n^{-1}$ has at most $2N+1$ nonzero entries in any row. Within each row, each entry is bounded by $C(\delta,N)$, so $(2N+1)C(\delta,N)$ is an upper bound on the largest eigenvalue of $\Sigma_n^{-1}$. This provides a positive lower bound on the smallest eigenvalue of $\Sigma_n$, and concludes the verification of the conditions for Theorem \ref{thm:density}, allowing us to conclude that these $X_i$ satisfy (P5*).

The argument used to show that $X_j$ is independent of $X_i$ given the other variables when $|i-j| > N$ works even when there are infinitely many other variables. Therefore it verifies (P*) with respect to the graph on $\mathbb{N}$ which has an edge between $i$ and $j$ if and only if $|i-j|\leq N$. Thus, by Theorem \ref{thm:CIREquiv}, we may conclude that the collection of variables $X_i$ satisfies (G*) with respect to this graph as well. \qed

\noindent \textbf{Details of Example \ref{exa:dd}.}
We demonstrate that the desired conditions are satisfied for these covariance functions, beginning by verifying that the $\sigma_{ij}$ are diagonally dominant for all sufficiently small $V$ (depending on $\alpha$). For any fixed $p = c(i) \in \mathbb{Z}^m$, and any choice of $\alpha$ and $V$,
\be \sum_{q \in \mathbb{Z}^m} \exp(-d(p,q)^\alpha/V) \leq 1 + \sum_{n \in \mathbb{N}} (2n+1)^m\exp(-n^\alpha/V) < \infty, \nonumber 
\ee
and moreover for all $q \neq p$, the quantity $\exp(-d(p,q)/V)$ tends to zero as $V$ tends to zero. Thus, by the dominated convergence theorem,
\be \sum_{j \neq i} |\sigma_{ij}(V)| \leq \sum_{q \neq p} \exp(-d(p,q)/V) \rightarrow 0 < 1 = \sigma_{ii}(V). \ee
Because the sum above (as a function of $V$) is independent of $p$, we may conclude that for $V$ sufficiently small (again, depending on the fixed choice of $\alpha$), the resulting covariance matrix is diagonally dominant. 

We omit the verification of the bounds on the $g_r$ for these covariances, but note that the verification is similar to that used to obtain line \eqref{ln:gbnd}. 

Thus, for sufficiently small positive $V$, the Gaussian processes with covariance functions given by
$$\sigma_{ij} = \exp(-d(c(i),c(j))^\alpha/V)$$
satisfy the requirements of Theorem \ref{thm:density}, and so satisfy (IIP) and (DCP), and hence (P5*). \qed \ \\

\noindent \textbf{Proof of Example \ref{exa:discmc}.}
For this proof, we will apply Theorem \ref{thm:discreteMR}. For any finite length $m$ sequence $(i_1,...,i_m)$ disjoint from $b_1,...,b_r$, we have
\bea 
& & \mathbb{P}((X_{i_1},...,X_{i_m}) = (a_1,...,a_m) \ | \ X_{j_1}=b_1,...,X_{j_r}=b_r) \nonumber \\
& \geq &  \prod_{j=1}^m \min_{y,z \in \{0,1\}}(\mathbb{P}(X_{i_j}=a_j|X_{i_j-1}=y, X_{i_j+1}=z )) \nonumber \\
& =: & c_I > 0. \nonumber 
\eea

Next, suppose that $n \in \N$ and $B \subseteq \mathbb{N}$ is arbitrary, and that $A$ is an event in $\sigma(X_B,X_n,X_{n+1},...)$. First, suppose that $B$ is infinite. Then, for any $m$, there is some $b \in B$ with $b > m$, so for all sufficiently large $n$ we have $m < b < n$. Thus
$$\mathbb{P}(A|X_1,...,X_m,X_B=x_B) = \mathbb{P}(A|X_B=x_B)$$
is constant, and so 
$$\var(\mathbb{P}(A|X_1,...,X_m,X_B)|X_B=x_B) = 0$$
for all sufficiently large $n$.
On the other hand, if $B$ is finite, we have that there is some maximal $b \in B$. If $b \geq m$, the above argument implies that 
$$\var(\mathbb{P}(A|X_1,...,X_m,X_B)|X_B=x_B) = 0$$
for all sufficiently large $n$, and so we may assume $m > b$.
Since the $X_i$ form a Markov chain, and since $X_m$ can only take values in $\{0,1\}$, we have that if $n > m$,
\bea & & \var(\mathbb{P}(A|X_1,...,X_m,X_B)|X_B=x_B) \nonumber \\
& =    & \var(\mathbb{P}(A|X_m, X_B)|X_B=x_B) \nonumber \\
& \leq & |\mathbb{P}(A|X_m=1, X_B = x_B)-\mathbb{P}(A|X_m=0, X_B = x_B)|. \label{ln:variancebound} \eea
For any arbitrary $m \leq r < n$, we have the bound
\bea & & |\mathbb{P}(A|X_r=1, X_B = x_B)-\mathbb{P}(A|X_r=0, X_B = x_B)| \nonumber \\
& =    & |\mathbb{P}(A|X_{r+1} = 1, X_B = x_B)\mathbb{P}(X_{r+1}=1|X_r=1) \nonumber \\
&      &+\ \mathbb{P}(A|X_{r+1} = 0, X_B = x_B)\mathbb{P}(X_{r+1}=0|X_r=1) \nonumber \\
&      & -\ \mathbb{P}(A|X_{r+1} = 1, X_B = x_B)\mathbb{P}(X_{r+1}=1|X_r=0) \nonumber \\
&      & -\ \mathbb{P}(A|X_{r+1} = 0, X_B = x_B)\mathbb{P}(X_{r+1}=0|X_r=0)| \nonumber \\
& = & |\mathbb{P}(A|X_{r+1} = 1, X_B = x_B)(\mathbb{P}(X_{r+1}=1|X_r=1)- \mathbb{P}(X_{r+1}=1|X_r=0)) \nonumber \\
&      & +\ \mathbb{P}(A|X_{r+1} = 0, X_B = x_B)(\mathbb{P}(X_{r+1}=0|X_r=1) - \mathbb{P}(X_{r+1}=0|X_r=0))| \nonumber \\
& = & |\mathbb{P}(A|X_{r+1} = 1, X_B = x_B)(p_r - t_r) + \mathbb{P}(A|X_{r+1} = 0, X_B = x_B)(t_r - p_r)| \nonumber \\
& = & |(\mathbb{P}(A|X_{r+1} = 1, X_B = x_B) - \mathbb{P}(A|X_{r+1} = 0, X_B = x_B))(p_r - t_r)| \nonumber \\
& = & |\mathbb{P}(A|X_{r+1} = 1, X_B = x_B) - \mathbb{P}(A|X_{r+1} = 0, X_B = x_B)||p_r-t_r|. \nonumber
\eea

Thus, by induction, we have that
\bea & & |\mathbb{P}(A|X_m=1, X_B = x_B)-\mathbb{P}(A|X_m=0, X_B = x_B)| \nonumber \\ & & \leq |\mathbb{P}(A|X_{n-1}=1, X_B = x_B)-\mathbb{P}(A|X_{n-1}=0, X_B = x_B)|\prod_{r=m}^{n-2} |p_r-t_r| \nonumber \\ & & \leq \prod_{r=m}^{n-2} |p_r-t_r|, \nonumber \eea
and so by line \eqref{ln:variancebound}, we have that 
$$\var(\mathbb{P}(A|X_1,...,X_m,X_B)|X_B=x_B) \leq \prod_{r=m}^{n-2} |p_r-t_r| =:  g_{m,B}(n).$$
Since $\sum_n (1-(p_n-t_n)) = \infty$, we have that $\prod_{r=m}^{n-2} |p_r-t_r| \rightarrow 0$ as $n \rightarrow \infty$, by the usual arguments involving the logarithm. We may now finish the proof by invoking Theorem \ref{thm:discreteMR}. \qed

The next two results are used to prove Proposition \ref{prop:ising}

 \begin{lem} \label{lem:isinglem}
Suppose given a graph $\mathcal{G} = (V,E)$ with $V = \{0\} \cup \mathbb{N}$ and parameter set $\Theta$. If the limit
$$\lim_{n \rightarrow \infty} \mathbb{P}^m_n((X_1,...,X_m)=v)$$
exists for all $v \in \{0,1\}^m$ and all $m \in \mathbb{N}$, the infinite Ising model distribution from Definition \ref{def:iim} exists and is unique.
\end{lem}

 \begin{proof}
Note that for all finite $n$ we have 
$$\sum_{v \in \{0,1\}^m} \mathbb{P}^m_n((X_1,...,X_m)=v) = 1,$$
and so
\bea 
&   & \sum_{v \in \{0,1\}^m} \mathbb{P}^m((X_1,...,X_m)=v) \nonumber \\
& = & \sum_{v \in \{0,1\}^m} \lim_{n\to\infty} \mathbb{P}^m_n((X_1,...,X_m)=v) \nonumber \\
& = & \lim_{n\to\infty} \sum_{v \in \{0,1\}^m} \mathbb{P}^m_n((X_1,...,X_m)=v) \nonumber \\
& = & \lim_{n\to\infty} 1 = 1. \nonumber
\eea
Thus the distributions $\mathbb{P}^m$ are actually probability distributions on $\{0,1\}^m$. Moreover, by the properties of marginalization, for $m_1 < m_2$, we have for all sufficiently large $n$, and any $v \in \{0,1\}^{m_1}$, that
$$\mathbb{P}^{m_2}_n((X_1,...,X_{m_1}) = v) = \mathbb{P}^{m_1}_n((X_1,...,X_{m_1}) = v),$$
and so
$$\mathbb{P}^{m_2}((X_1,...,X_{m_1}) = v) = \mathbb{P}^{m_1}((X_1,...,X_{m_1}) = v),$$
and so the marginalizations of the $\mathbb{P}^m$ are consistent. Thus by the Kolmogorov Consistency theorem, there exists a unique joint distribution $\mathbb{P}$ on $\{0,1\}^\infty$ with finite-dimensional distributions given by the distributions $\mathbb{P}^m$ and their marginalizations.
\end{proof}
 \ \\

We now state a more explicit condition which can be used to verify whether the limit in line \eqref{ln:limexist} exists.

\begin{prop} \label{prop:fmvn}
Suppose given a graph $\mathcal{G} = (\{0\}\cup\mathbb{N},E)$ and parameter set $\Theta$. For $v \in \{0,1\}^m$, and $n \geq m$, define the function 
\be f_m(v,n) := 
\frac{\sum_{X \in \{0,1\}^{n-m}}\exp\left(\sum_{i,j > m} \theta_{ij} X_i X_j + \sum_{i\leq m, j > m} \theta_{ij} v_i X_j\right)}
{\sum_{X \in \{0,1\}^{n-m}}\exp\left(\sum_{i,j > m} \theta_{ij} X_i X_j + \sum_{j > m} \theta_{0j} X_j\right)} \label{ln:fmvn} \ee
where for $X \in \{0,1\}^{n-m}$ we have labeled the first coordinate as $X_{m+1}$ and the final coordinate $X_n$, and for ease of notation we have used $v_0 = 1$.

Then the limit in line \eqref{ln:limexist} exists and is nonzero for all $m$ and $v$ if and only if the limit $\lim_{n \to \infty} f_m(v,n)$ exists and is nonzero for all $m$ and $v$.
\end{prop}

\begin{proof}
We have
\be f_m(v,n) = \frac{\mathbb{P}^m_n((X_1,...,X_m) = (v_1,...,v_m))}{\mathbb{P}^m_n((X_1,...,X_m) = (0,...,0))}. \label{ln:fmvnalt} \ee
From this, it is easy to see that if the limit from line \eqref{ln:limexist} exists and is finite and positive for all $m$ and $v$, then the same holds of $\lim_{n \to \infty} f_m(v,n)$. 

For the other direction, note by line \eqref{ln:fmvnalt} that for any $n$ we have
$$\mathbb{P}^m_n((X_1,...,X_m) = (v_1,...,v_m)) = f_m(v,n)\mathbb{P}^m_n((X_1,...,X_m) = (0,...,0)).$$

\noindent Let $F_m(n) = \sum_{v \in \{0,1\}^m} f_m(v,n)$. Then $\mathbb{P}^m_n((X_1,...,X_m) = (0,...,0)) = 1/F_m(n)$. Assuming that $f_m(v,n)$ has a finite, nonzero limit as $n \to \infty$ for all $v$ and $m$, we obtain that $F_m(n)$ does as well, and so 
$$\lim_{n \to \infty} 1/F_m(n) = \lim_{n \to \infty} \mathbb{P}^m_n((X_1,...,X_m) = (0,...,0)) =: \mathbb{P}^m((X_1,...,X_m) = (0,...,0))$$
exists and is nonzero as well. Finally, note that
$$\mathbb{P}^m((X_1,...,X_m) = v) := \lim_{n \to \infty} \mathbb{P}^m_n((X_1,...,X_m) = v) = \lim_{n \to \infty} f_m(n,v) / F_m(n)$$
exists and is nonzero since $f_m(n,v)$ and $F_m(n)$ have limits as $n$ tends to $\infty$.
\end{proof}\ \\

\noindent \textbf{Proof of Proposition \ref{prop:ising}.} If the infinite Ising model is well-defined with $\mathbb{P}((X_1,...,X_m) = v) > 0$ for all $m, v$, then by the definition of the infinite Ising model the quantity
$$f_m(v,n) = \frac{\mathbb{P}^m_n((X_1,...,X_m) = (v_1,...,v_m))}{\mathbb{P}^m_n((X_1,...,X_m) = (0,...,0))}$$
has a nonzero limit for all $m$ and $v$, and so the equivalent expression from line \eqref{ln:fmvn} does as well. For the other direction, combine Lemma \ref{lem:isinglem} with Proposition \ref{prop:fmvn}.\qed

\end{document}